\documentclass[sn-mathphys,Numbered]{sn-jnl}

\usepackage{graphicx,xcolor}%
\usepackage{multirow}%
\usepackage{amsmath,amssymb,amsfonts, upgreek}%
\usepackage{amsthm,nicefrac}%
\usepackage{mathrsfs}%
\usepackage[title]{appendix}%
\usepackage{xcolor}%
\usepackage{textcomp}%
\usepackage{manyfoot}%
\usepackage{booktabs}%
\usepackage[vlined, ruled, algo2e, linesnumbered]{algorithm2e}
\usepackage{listings}%
\usepackage{bm}
\usepackage{enumitem}
\usepackage{tabularx}
\usepackage{siunitx}

\usepackage{subcaption}
\usepackage{float}
\usepackage{placeins}

\usepackage{bm}
\usepackage{bbm}
\usepackage{mathtools}   
\usepackage{accents}
\usepackage{graphicx} 
\usepackage{comment}
\usepackage{xcolor}
\usepackage{multirow}
\usepackage{hyperref}
\usepackage{cleveref}

\usepackage{tikz}
\usepackage{pgfplots}
\pgfplotsset{compat=newest} 
\usepgfplotslibrary{groupplots}

\DeclareMathOperator*{\argmin}{arg\,min}

\newcommand{\R}{\mathbb R}

\makeatother
\def\theequation{\thesection.\arabic{equation}}
\newcommand{\bx}{{\bm x}}

\newcommand{\yd}{{y^{\delta}}}

\newcommand{\Q}{{\mathscr Q}}

\newcommand{\Csp}{{\mathscr C}}
\newcommand{\Qad}{{\mathscr Q_\mathsf{ad}}}
\newcommand{\qad}{{Q_{\mathsf{ad},h}}}
\newcommand{\qadK}{{Q_{\mathsf{ad},h}^K}}
\newcommand{\qadr}{{Q_{\mathsf{ad},r}}}

\newcommand{\qadri}{{Q^\iteridx_{\mathsf{ad},r}}}

\newcommand{\Sol}{\mathcal S}
\newcommand{\C}{{\mathcal C}}
\newcommand{\F}{{\mathcal F}}

\newcommand{\ulin}{\tilde{u}}
\newcommand{\plin}{\tilde{p}}
\newcommand{\iteridx}{{(i)}}
\newcommand{\errest}{\mathcal{R}^\iteridx}

\newcommand{\qtrial}{q^\iteridx_{\mbox{\tiny{trial}},r}}

\newcommand{\epsPODstate}{\epsilon_\text{POD,V}}

\newcommand{\qex}{q^\mathsf e}
\newcommand{\termidx}{i^*(\delta,\yd)}
\newcommand{\termidxh}{i_h^*(\delta,\yd)}

\newcommand{\errprim}[2]{#2{e}^{#1}}
\newcommand{\errprimdot}[2]{\dot{#2{e}}^{#1}}
\newcommand{\stateerrorest}[1]{\Delta^{\text{pr}}({#1})}
\newcommand{\tildestateerrorest}[1]{\tilde{\Delta}^{\text{pr}}({#1})}
\newcommand{\objerrorest}[2]{\Delta^{J}({#1, #2})}
\newcommand{\objerrorestsymb}{\Delta^{J}}
\newcommand{\opsensors}{\mathcal{C}_\text{sensors}}
\newcommand{\opall}{\mathcal{C}_\text{id}}
\newcommand{\IBVP}{\eqref{eq:ibvp_pde}--\eqref{eq:ibvp_initial}}
\newcommand{\relnoise}{\delta_{\mathrm{rel.}}}

\newcommand{\trace}[1]{\mathrm{tr}(#1)}
\newcommand{\frobprod}[2]{\langle\langle#1|#2\rangle\rangle}

\definecolor{DarkRed}{rgb}{0.5 0 0}
\definecolor{DarkBlue}{rgb}{0 0 0.5}
\definecolor{DarkGreen}{rgb}{0 0.5 0}
\definecolor{DarkOrange}{rgb}{0.83 0.33 0}
\definecolor{DarkMagenta}{rgb}{0.83 0 0.67}

\usepackage[acronym,nonumberlist]{glossaries}
\newacronym{SHM}{SHM}{Structural Health Monitoring}
\newacronym{PDE}{PDE}{Partial Differential Equation}
\newacronym{MOR}{MOR}{Model order reduction}
\newacronym{RB}{RB}{reduced bases}
\newacronym{IRGNM}{IRGNM}{Iteratively Regularized Gauss Newton Methods}
\newacronym{POD}{POD}{proper orthogonal decomposition}
\newacronym{FOM}{FOM}{full-order model}
\newacronym{ROM}{ROM}{reduced-order model}
\newacronym{TR}{TR}{trust-region}
\newacronym{ML}{ML}{machine learning}
\newacronym{HaPOD}{HaPOD}{hierarchical approximate POD}
\glsdisablehyper

\newtheoremstyle{Stil}{}{}{\itshape}{}{\bfseries}{.\\}{1ex}{}
\newtheoremstyle{Stil2}{}{}{\rmfamily}{}{\bfseries}{.}{1ex}{}
\theoremstyle{Stil}
\newtheorem{theorem}{Theorem}[section]

\theoremstyle{plain}
\newtheorem{remark}[theorem]{Remark}
\newtheorem{assumption}[theorem]{Assumption}

\theoremstyle{Stil2}

\crefname{assumption}{\textup{Assumption}}{\textup{Assumptions}}
\crefname{lemma}{\textup{Lemma}}{\textup{Lemmas}}
\crefname{theorem}{\textup{Theorem}}{\textup{Theorems}}
\crefname{remark}{\textup{Remark}}{\textup{Remarks}}
\crefname{example}{\textup{Example}}{\textup{Examples}}
\crefname{corollary}{\textup{Corollary}}{\textup{Corollaries}}
\crefname{subsection}{\textup{Section}}{\textup{Subsections}}
\crefname{section}{\textup{Section}}{\textup{Sections}}
\crefname{figure}{\textup{Figure}}{\textup{Figures}}
\crefname{table}{\textup{Table}}{\textup{Tables}}

\begin{document}

\title[Article Title]{
Adaptive Reduced-Basis Trust-Region Methods for Defect Identification in Elastic Materials
}


\author*[1]{\fnm{Benedikt} \sur{Klein}}\email{benedikt.klein@uni-muenster.de}

\author[1]{\fnm{Mario} \sur{Ohlberger}}\email{mario.ohlberger@uni-muenster.de}%

\author[2]{\fnm{Thomas} \sur{Schuster}}\email{thomas.schuster@num.uni-sb.de}%

\affil[1]{\orgdiv{Applied Mathematics: Institute for Analysis and Numerics}, \orgname{University of M\"unster}, \orgaddress{\street{Einsteinstr. 62}, \city{M\"unster}, \postcode{48149},  \country{Germany}}}
\affil[2]{\orgdiv{Department of Mathematics}, \orgname{Saarland University}, \orgaddress{\street{Campus}, \city{Saarbrücken}, \postcode{66123},  \country{Germany}}}


\date{\today}

\abstract{
Monitoring the integrity of elastic structures using ultrasonic waves requires the efficient identification of material parameters from measured surface displacements. The displacement field is governed by Cauchy's equation of motion, i.e., an elastic wave equation. Consequently, defect localization leads to a high-dimensional spatial parameter identification problem for a hyperbolic system with given initial and boundary conditions. Stable parameter reconstructions typically rely on regularization techniques such as the iteratively regularized Gauss--Newton method (IRGNM). However, its practical application is computationally demanding due to the high-dimensional nature of the problem.

To address this bottleneck, we propose a reduced-order modeling approach that simultaneously reduces the state and parameter spaces using adaptively constructed reduced-basis spaces. This yields online-efficient surrogate models for both the forward and adjoint evaluations required in derivative-based optimization. To ensure reliability, the IRGNM iteration is embedded into an adaptive, trust-region framework that provides accuracy of the reduced-order approximations. The approach extends our recent contributions in \cite{kartmann_adaptive_2023, kartmann2025adaptivereducedbasistrust}, which focus on elliptic and parabolic problems, to the hyperbolic setting. We demonstrate the reliability and effectiveness of the method for defect detection through numerical experiments.
}
\date{\today}

\keywords{parameter identification, model reduction, inverse problems, hyperbolic PDEs, Gauss-Newton methods}


\pacs[MSC Classification]{35R30, 35L10, 65M32}

\maketitle
\section{Introduction}
\label{sec:intro}

Detecting defects or impurities in (hyper)elastic materials is crucial in many industrial sectors, for example, in aerospace and biomedicine. In recent decades, a variety of \acrfull{SHM} systems have been developed, enabling non-destructive defect detection. These systems typically consist of a set of actuators and sensors attached to the surface of the structure, cf.~\cite{giurgiutiu2014structural,LAMBWAVES-BUCH:18}. The basic idea is that actuators generate waves that propagate through the material and interact with defects, for example, through reflection and scattering, before being measured by sensors. Based on these measurements, spatially distributed material parameters are reconstructed to infer the presence and location of defects. In this work, we restrict our attention to the reconstruction of the stored energy function of a linear elastic material characterized by a spatially distributed parameter field. In particular, our objective is to recover spatial variations of the Lamé parameters relative to a homogeneous background to identify deviations indicative of defects.

\textbf{Defect detection in elastic materials.}
We briefly review related approaches for defect detection and parameter identification in elasticity. Dynamic Load Monitoring (DLM), i.e., the reconstruction of time-dependent loads from measurement data in order to infer structural damage, constitutes an important class of inverse problems in structural health monitoring; see, e.g., \cite{YANG;ET;AL:20,JADAMBA;ET;AL:17, sun2020time, SOMAN;OSTACHOWICZ:19, PATEL;ET;AL:19, IELE;ET;AL:18, TASHAKORI;ET;AL:16}. General overviews on inverse elasticity and elasticity imaging are given in \cite{BONNET:05,fedele2024review, jankowski2009off, zhang2010identification}. Classical model order reduction approaches applied to structural health monitoring have been investigated in \cite{MR3736474,MR4064816} and recently generalized to a multi-fidelity surrogate modeling approach in \cite{Torzoni2023}.
In recent years, machine-learning-based approaches for the identification of elastic material parameters and constitutive relations have gained increasing attention, cf.~\cite{chen2023physics,ni2021deep,bhattacharya2025learning,bhattacharya2026optimal, Ghajari_2013}. At the same time, iterative methods for nonlinear inverse parameter identification in hyperelastic materials have been investigated in \cite{klein2021sequential,seydel2017identifying,seydel2017linearization}. 
In this setting, wave propagation is described by nonlinear hyperbolic systems, for which uniqueness results have been established in \cite{woestehoff2015uniqueness}. 

A major challenge shared by many of these approaches is the high computational cost arising from the repeated solution of large-scale forward problems during the iterative reconstruction process. In particular, accurate simulations of wave propagation phenomena require fine finite element discretizations and, consequently, large computational resources. Consequently, practical implementations often rely either on coarse discretizations, which may deteriorate reconstruction quality, or on early stopping of the iterative schemes, leading to insufficient parameter reconstructions. The present work addresses this challenge by combining iterative parameter identification with adaptive reduced-order modeling techniques to significantly reduce computational complexity while maintaining reconstruction accuracy.

\textbf{Trust-region reduced-basis parameter identification.}
In the present work, we focus on iteratively regularized Gauss--Newton-type methods (IRGNM) for inverse parameter identification problems. Such methods have been extensively studied from both mathematical and numerical perspectives, see, e.g., \cite{Hanke_1997,MR2337589,KaltenbacherMain,KKV14,KP18,kaltenbacher2021tangential,kaltenbacher2021time,kaltenbacher2021parameter}. To alleviate the computational challenges discussed above, we incorporate \acrfull{MOR} techniques into the iterative reconstruction procedure. Reduced-order modeling approaches for accelerating inverse problems have previously been investigated in \cite{lieberman2010parameter,HimpeOhlberger2015,gubisch2017proper,GhattasWillcox2021,yao2025derivative}. The main idea of MOR is to replace the high-dimensional full-order model by a low-dimensional surrogate that accurately approximates the forward and, if required, the associated adjoint problems. General introductions to \acrshort{MOR} can be found in~\cite{MR3672144,quarteroni2015reduced}. 

The present work extends the adaptive trust-region reduced-basis framework introduced in \cite{kartmann_adaptive_2023,kartmann2025adaptivereducedbasistrust} to inverse parameter identification problems constrained by Cauchy’s equation of motion from continuum mechanics, cf.~\cite{holzapfel2002nonlinear}. A central component of the approach is the iterative construction of low-dimensional reduced spaces using snapshot-based techniques such as \gls{POD}~\cite{Haasdonk2008,Haasdonk13,himpe2018hierarchical}. In contrast to classical reduced-basis methods relying on globally constructed approximation spaces, the reduced models are generated adaptively within the trust-region iteration. More precisely, solution snapshots corresponding to the current parameter iterates are used to incrementally enrich the reduced spaces, yielding surrogate models that accurately capture the local behavior of the full-order system. The trust-region framework controls the optimization by restricting parameter updates to regions in which the reduced-order model provides sufficiently accurate approximations. In this way, the proposed approach significantly reduces computational cost while maintaining reconstruction accuracy for large-scale inverse problems in elasticity.

\textbf{Main result.} This work introduces a novel simulation framework for defect identification in elastic materials by applying error aware trust-region methods combined with reduced-basis surrogates to inverse problems governed by hyperbolic PDEs.
The proposed approach combines parameter and state reduction for the reconstruction of spatially distributed parameter fields, demonstrated by the identification of the stored energy function of an elastic material.
 A central contribution of this work is the extension of the \acrlong{TR}-\acrlong{IRGNM} framework to inverse problems constrained by hyperbolic PDEs. In this context, we derive an a posteriori error estimator for the reduced-order model. Although the estimator is found to be overly conservative for direct use within the optimization loop, it provides a rigorous error control mechanism for the reduced formulation. Moreover, we consider observation operators representing realistic sensor measurements rather than assuming full-state observations. Numerical experiments demonstrate that the proposed approach significantly reduces computational cost while maintaining reconstruction accuracy compared to standard finite-element-based optimization methods. The numerical experiments are enabled by a substantial extension of the existing \texttt{pyMOR}-based software framework \cite{doi:10.1137/15M1026614}, including support for full-order models implemented in \texttt{deal.II} \cite{2024:africa.arndt.ea:deal}.

\textbf{Organization of the article.}
In Section~\ref{ssec:problem_formulation}, the problem of identifying material parameters in (hyper)elastic materials is formulated. In particular, we specify the assumptions on the stored energy function that guarantee well-posedness of the resulting inverse problem. In Sections~\ref{subsec:IRGNM} and \ref{ssec:realization_IRGNM}, the IRGNM formulation for this inverse problem is developed, and the corresponding linearized primal and adjoint problems required for the numerical solution are stated. Section~\ref{ssec:discretization} presents the discretization of the governing equations as a first-order system using a generalized implicit time-stepping scheme. Based on this discretization, reduced-basis surrogates are derived in Section~\ref{ssec:modelreduction}, where the full-order state and parameter spaces are replaced by their reduced counterparts. In particular, the strategy for constructing these spaces is introduced, and an a posteriori error estimator is derived. Section~\ref{sec:TR_IRGNM} briefly recalls the TR-IRGNM method introduced in \cite{kartmann2025adaptivereducedbasistrust} and outlines the integration of the reduced-basis \acrlong{ROM}s into it. Finally, numerical examples are presented in Section~\ref{sec:num_exper}, followed by concluding remarks and a discussion of the numerical results.

\section{Inverse Problem for Hyperelastic Wave Equations}
\label{sec:IP_for_wave}

\subsection{Parametrized Forward Problem}
\label{ssec:problem_formulation}

In this work, we consider inverse problems for Cauchy's equation in the theory of elastic materials.
\begin{equation*}
    \rho(\bx)\, \partial_t^2 u(t,\bx)
    - \nabla \cdot \nabla_Y C_q\big(\bx, Ju(t,\bx)\big)
    = l(t,\bx),
\end{equation*}
posed in a bounded open Lipschitz domain $\Omega \subset \mathbb{R}^3$ over a time interval $[0,T]$. Here, $\rho \in L^\infty(\Omega, \mathbb{R}_{> 0})$ denotes the mass density, and $l : [0,T] \times \Omega \to \mathbb{R}^3$ is the external body force. The function $C = C(\bx, Y)$, with $Y \in \mathbb{R}^{3 \times 3}$ and $\det Y > 0$, is called the stored energy function of the material. The derivative $\nabla_Y$ is taken with respect to the components of $Y$, and $Ju(t,\bx)$ denotes the Jacobian matrix of $u(t,\bx)$. For given initial data $u_0 \in H^2(\Omega; \mathbb{R}^3)$ and $u_1 \in H^1(\Omega; \mathbb{R}^3)$, we consider mixed homogeneous boundary conditions on $\partial \Omega$. Let $\Gamma \subset \partial \Omega$ be a nonempty, relatively open subset, with positive surface measure. On $\Gamma$, the body is fixed, while it can move freely on $\partial \Omega \setminus \overline{\Gamma}$. Thus, the displacement field $u : [0,T] \times \Omega \to \mathbb{R}^3$ satisfies the following initial--boundary value problem (IBVP):
\begin{align}
    \rho(\bx)\, \partial_t^2 u(t,\bx)
    - \nabla \cdot \nabla_Y C_q\big(\bx, Ju(t,\bx)\big)
    &= l(t,\bx)
    && \text{in } (0,T) \times \Omega,
    \label{eq:ibvp_pde} \\
    u(t,\bx) &= 0
    && \text{on } (0,T) \times \Gamma,
    \label{eq:ibvp_dirichlet} \\
    \nabla_Y C_q\big(\bx, Ju(t,\bx)\big) \cdot n &= 0
    && \text{on } (0,T) \times (\partial\Omega \setminus \overline{\Gamma}),
    \label{eq:ibvp_neumann} \\
    u(0,\bx) = u_0, \quad \partial_t u(0,\bx) &= u_1
    && \text{in } \Omega.
    \label{eq:ibvp_initial}
\end{align}

To obtain a parametrized formulation of the inverse problem, we adopt a structured representation of the stored energy function inspired by \cite{seydel2017linearization, seydel2017identifying, woestehoff2015uniqueness, binder2015defect}. Specifically, we assume that $C_q(x,Y)$ admits an affine decomposition with respect to a finite-dimensional parameter vector $q$, and satisfies suitable regularity, symmetry, and a Gårding-type coercivity conditions. This setting enables a reduced representation of the material law and provides the foundation for the proposed reconstruction method.

\begin{assumption}[Adapted from \cite{seydel2017linearization, seydel2017identifying}]
\label{ass:ass_to_stored_energy}
Let the stored energy function be given by
\begin{equation}
\label{eq:linear_decomp_stored_energy}
    C_q(\bm x, Y) = C_{0}(\bm x, Y) + \sum_{p = 1}^{N_Q} q_p \, C_{p}(\bm x, Y),
\end{equation}
where the parameter vector $q \coloneqq (q_1, \dots, q_{N_\Q})$ is identified as an element in the discrete parameter space
\begin{equation*}
\Q \coloneqq \mathbb{R}^{N_Q}.
\end{equation*}
Assume that for all $q \in \Qad$, we have $C_q \in \mathcal{C}^4\big(\Omega \times \mathbb{R}^{3\times 3}\big)$ and, the map 
$Y \mapsto \nabla_Y C_q(\bm x, Y)$ is \emph{linear} for all $\bm x \in \Omega$. Moreover, we impose the \emph{major symmetry} condition
\begin{equation}
\label{eq:major_symmetry}
    \partial_{Y_{ij}} \partial_{Y_{kl}} C_{p}(\bm x, Y) 
    = \partial_{Y_{kl}} \partial_{Y_{ij}} C_{p}(\bm x, Y),
\end{equation}
to hold for all $p = 0, \dots, N_Q$, $i,j,k,l = 1,2,3$, and $(x,Y) \in \Omega \times \mathbb{R}^{3\times 3}$. Additionally, we demand that for each $p = 1, \dots, N_Q$, there exist constants $\kappa_p, \mu_p > 0$ such that
\begin{equation}
\label{eq:coercivity}
    \kappa_p \, \|Y\|_F^2 \leq \nabla_YC_p(\bm x,Y) : Y \leq \mu_p \, \|Y\|_F^2
\end{equation}
for all $x\in \Omega$ and $Y \in \mathbb{R}^{3 \times 3}$ symmetric, i.e., $Y^\top = Y$, where $\|\cdot\|_F$ denotes the Frobenius norm induced by the inner product
\begin{equation*}
    \frobprod{A}{B} \coloneqq \trace{A^\top B} \text{ for } A,B \in \mathbb{R}^{3 \times 3}.
\end{equation*}
We denote $\kappa \coloneqq (\kappa_0, \dots, \kappa_{N_Q})$, $\mu \coloneqq (\mu_0, \dots, \mu_{N_Q})$ and define the set of \emph{admissible parameters} $\Qad$ such that
\begin{equation}
\label{eq:ad_param_set}
    \Qad \subset \left\{
        q \in \Q \;\middle|\;
        q_p > 0 \ \forall p, \;
        q^\top \kappa \geq \tilde{\kappa},\;
        q^\top \mu \leq \tilde{\mu}
    \right\},
\end{equation}
for some fixed constants $\tilde{\kappa} > 0$ and $\tilde{\mu} > 0$. We furthermore assume that the mass density is homogeneous, i.e. $\rho(x) = \rho > 0$ for all $x \in \Omega$.
\end{assumption}

We consider the weak formulation of \IBVP. To this end, we introduce the vector-valued spaces
\begin{equation*}
    V \coloneqq H^1_\Gamma(\Omega; \mathbb{R}^3)
    \coloneqq
    \{
        v \in H^1(\Omega; \mathbb{R}^3) \mid v|_\Gamma = 0
    \}
    \quad \text{and} \quad
    H \coloneqq L^2(\Omega; \mathbb{R}^3),
\end{equation*}
which form the Gelfand triple $V \hookrightarrow H \hookrightarrow V'$. 
We define the operator $A(q):V \to V'$ by
\begin{equation}
\label{eq:def_bilinear_form}
    \langle A(q)u, v \rangle_{V',V}
    \coloneqq \int_\Omega \nabla_Y C_q\big(\bx, Ju(\bx)\big) : Jv(\bx) \, d\bx,
    \quad \forall u,v \in V,
\end{equation}
we obtain the weak formulation of \IBVP. 
For $l \in H^1(0,T; H)$, $u_0 \in H^2(\Omega; \mathbb{R}^3)$ satisfying the compatibility conditions
\begin{equation*}
    u_0 = 0 \ \text{on } \Gamma, 
    \quad 
    \nabla_Y C_q\big(\bx, Ju_0(\bx)\big) \cdot n = 0 \ \text{on } \partial\Omega \setminus \overline{\Gamma},
\end{equation*}
$u_1 \in V$, and $q \in \Qad$, find $u$ such that
\begin{equation}
\label{eq:weak_IBVP}
\begin{aligned}
    \langle \rho \partial_t^2 u(t) + A(q)u(t), v \rangle_{V',V} &= \langle l(t), v \rangle_{V',V},\\
    u(0) = u_0, \
    \partial_t u(0) &= u_1,
\end{aligned}
\end{equation}
for all $v \in V$ and for almost all $t \in (0,T)$. 
We seek a solution $u$ such that
\begin{equation*}
    u \in W(0,T) \coloneqq H^1(0,T; V) \cap H^2(0,T; H).
\end{equation*}
Under Assumption~\ref{ass:ass_to_stored_energy}, problem \eqref{eq:weak_IBVP} admits a unique solution in $W(0,T)$ for sufficiently smooth $\partial \Omega$, $u_0$, $u_1$, and $l$; cf. \cite[Lemma 2.3 \& Corollary 2.5]{binder2015defect} and \cite[Theorem 30.4]{wloka1987partial}. Note that, for all $q \in \Qad$, the operator $A(q) \in \mathcal{L}(V, V')$ induces a continuous, symmetric bilinear form satisfying
\begin{equation*}
\label{eq:garding}
    \langle A(q) v, v \rangle_{V',V} + \alpha_q \|v\|^2_H \geq \alpha_q \|v\|^2_V,
\end{equation*}
with $\alpha_q = q^\top \kappa$. We therefore define the solution operator
\begin{equation*}
    \Sol : \Qad \to W(0,T), 
    \quad u(q) \coloneqq \Sol(q),
\end{equation*}
and refer to $u(q)$ as the \emph{primal state}. We assume that $\Sol$ is continuously Fréchet differentiable. 
This holds under additional regularity assumptions on $\partial\Omega$, $u_0$, $u_1$, and $l$, as well as further structural conditions on $C_q$ beyond Assumption~\ref{ass:ass_to_stored_energy}; cf.~\cite{seydel2017linearization}.

In general, the state $u(q)$ is not directly observable through physical measurements. We therefore introduce an observation space
\begin{equation*}
    \Csp \coloneqq L^2(0,T; C), 
\end{equation*}
where $C$ is a real Hilbert space. We consider an injective, linear, and bounded observation operator $\mathcal{C} \in \mathcal{L}(L^2(0,T;V), \Csp)$. For simplicity, we assume that $\mathcal{C}$ is time-independent. Thus, it can equivalently be interpreted as an operator $\mathcal{C} \in \mathcal{L}(V,C)$ acting pointwise in time via 
$(\mathcal{C}u)(t) = (\mathcal{C}u)(t)$ for all $u \in L^2(0,T;V)$ and $t \in (0,T)$.
A suitable observation operator is, for instance, the canonical embedding
\begin{equation}
  \label{eq:full_data_obs_op}
  \opall : W(0,T) \hookrightarrow L^2\!\left(0,T;V\right),
  \qquad
  \opall v = v,
\end{equation}
which corresponds to full knowledge of the displacement field. Note that 
$\opall$ is a bounded linear operator. This follows, for example, from the Poincaré inequality. In practical applications, however, measurements are usually only available on the surface 
of the material. Following \cite{binder2015defect,seydel2017identifying}, we model the 
measurement process by sensors that average the displacement field over small portions of 
the boundary $\partial\Omega$. To this end, we define the operator
\begin{equation*}
  \mathcal{C}_{\mathrm{sens}} : W(0,T) \rightarrow L^2\!\left(0,T;\mathbb{R}^{N_c}\right)
\end{equation*}
by
\begin{equation}
\label{eq:sensor_data_obs_op}
\begin{aligned}
\bigl(\mathcal{C}_{\mathrm{sens}} v\bigr)(t)
&= \left( \left\langle g_c,\, T[v(t)] \right\rangle_{L^2(\partial\Omega;\mathbb{R}^3)} \right)_{c=1,\dots,N_c} \\
&= \left( \int_{\partial\Omega} \langle g_c(\xi),\, v(t,\xi) \rangle_{\mathbb{R}^3}\, d\xi \right)_{c=1,\dots,N_c}.
\end{aligned}
\end{equation}
Here, $T : H^1(\Omega;\mathbb{R}^3) \rightarrow H^{1/2}(\partial\Omega;\mathbb{R}^3)$ 
denotes the trace operator, $N_c$ is the number of sensors, and 
$g_c \in L^2(\partial\Omega;\mathbb{R}^3)$, $c=1,\dots,N_c$, are weighting functions that localize the measurements to small regions of $\partial\Omega$.

The inverse problem is therefore stated as: Given $(l, u_0, u_1)$ and some noise-contaminated data $y^\delta \in \Csp$, a parameter vector $q \in \Qad$ has to be identified such that it best fits the data $y^\delta$, i.e., solving the minimization problem 
\begin{equation} 
\label{eq:minJ} 
\tag{$\mathbf{IP}$} 
\min J(q)\coloneqq \frac{1}{2} \,{\|\F(q)-\yd\|}_\Csp^2 \quad \text{subject to (s.t.) }\quad q\in\Qad, 
\end{equation}
where $\F:\Qad\to\Csp\quad\text{with }\F=\mathcal C\circ\mathcal S$ will be called the forward operator. Throughout the work, we assume that the measurement deviates from the exact data  $y^\mathsf e\in \text{range}(\F)$ at most by a known noise level $\delta>0$
\begin{align*} 
{\|y^\mathsf e-\yd\|}_\Csp\leq \delta, 
\end{align*}
and that a unique $q^\mathsf e \in \Qad$ exists, fitting exact data $y^\mathsf e$, i.e.
\begin{align} 
\label{eq: IP} 
\F(q^\mathsf e)=y^\mathsf e\quad\text{in }\Csp.
\end{align}

\subsection{Iteratively-regularized Gauss-Newton method (IRGNM)}
\label{subsec:IRGNM}
We aim to solve \eqref{eq:minJ} using the \emph{iteratively regularized Gauss-Newton method} (IRGNM), cf.~\cite{KKV14, kaltenbacher2009iterative}. During the IRGNM the current parameter $q^\iteridx \in \Qad$ is iteratively updated, by some update directions $d^\iteridx \in \Q$, such that 
\begin{equation*}
 q^{(i+1)} \coloneqq q^{\iteridx} + d^\iteridx \in \Qad,
\end{equation*}
starting from $q^{(0)} \in \Qad$, sufficiently close to the exact solution $\qex \in \Qad$. The update directions $d^\iteridx \in \Q$ are computed as solutions to  
\begin{equation}
	\label{eq:IRGNMscheme_minimize}
	\tag{$\widetilde{\textbf{IP}}_\alpha$}
	d^\iteridx \coloneqq \argmin\tilde{J}(d; q^\iteridx, \alpha^\iteridx)\quad \text{s.t.} \quad q^\iteridx + d \in\Qad,
\end{equation}
where $\tilde{J}$ denotes the linearization of the objective, regularized by a Tikhonov term
\begin{equation*}
    \tilde{J}(d; q^\iteridx, \alpha^\iteridx) \coloneqq  \frac{1}{2}\,{\|\F(q^\iteridx)+\F'(q^\iteridx)d-\yd\|}_\Csp^2 + \frac{\alpha^\iteridx}{2}\,{\|q^\iteridx + d -q_\circ \|}_\Q^2.
\end{equation*}
with regularization center $q_\circ\in\Q$. The regularization parameter $\alpha^\iteridx >0$ is chosen a posteriori at each iteration. More precisely, $\alpha^\iteridx$ is adjusted iteratively by repeatedly solving \eqref{eq:IRGNMscheme_minimize} until
\begin{equation}
	\label{eq: choice of alpha}
	\theta J(q^\iteridx)\leq 2\tilde{J}(d^\iteridx; q^\iteridx, 0)  =  {\|\F'(q^\iteridx)d^\iteridx+\F(q^\iteridx)-\yd\|}_\Csp^2 \leq \Theta J(q^\iteridx),
\end{equation}
is satisfied for $0<\theta<\Theta<2$. As stopping criteria, we use the so-called \emph{discrepancy principle}: the iteration is stopped after $i_*(\delta,\yd)$ steps provided
\begin{equation}\label{eq:discrepancy_principle}
	J(q^{\termidx})\leq \frac{1}{2}(\tau \delta)^2 \leq  J(q^\iteridx), \quad\ i=0,...,\termidx -1.
\end{equation}
The parameter $\tau>1$ reflects that we cannot expect the discrepancy to be lower than the noise in the given data.

\subsection{Realization of the IRGNM}
\label{ssec:realization_IRGNM}
For the implementation of a numerical procedure, efficient access to the derivatives of $J$ and $\tilde{J}$ is required. We will compute them by a standard adjoint argument. Let $A(q)'\in \mathcal{L}(V,V')$ denote the adjoint operator of $A(q)$, i.e. 
\begin{equation}
\label{eq:adjoint_A}
    \langle A(q)'p,u\rangle_{V',V}= \langle A(q)u,p\rangle_{V',V}
\end{equation}
for all $u,p\in V$ and $q\in \Q$. The adjoint equation associated with the objective in \eqref{eq:minJ} is given for $q\in \Qad$ as
\begin{align}
    \label{eq:adjoint_eq}
    \begin{aligned}
        \langle \rho\partial^2_t p(t)+A(q)'p(t), v \rangle_{V', V} & = - \langle \C u(t)-y^\delta(t), \C v \rangle_C \\
        p(T) &= 0, \ 
        \partial_t p(T) = 0 && \text{in } H,
    \end{aligned}
\end{align}
for all $v \in V$ and f.a.a. (for almost all) $t\in (0,T)$. Under Assumption \ref{ass:ass_to_stored_energy}, for every $u \in W(0,T)$ there exists a unique adjoint state
\begin{equation*}
    p \in H^1(0,T;H)\cap L^2(0,T;V),
    \qquad
    \partial_t^2 p \in L^2(0,T;V'),
\end{equation*}
satisfying \eqref{eq:adjoint_eq} f.a.a. $t \in (0,T)$. By Assumption~\ref{ass:ass_to_stored_energy}, the mapping ${q \mapsto A(q)}$ is affine and therefore Fréchet differentiable in $q$. It can thus be written as
\begin{equation}
\label{eq:A_affine}
A(q) = A_0 + \sum_{p=1}^{N_Q} q_p A_p,
\end{equation}
where $A_0, A_p \in \mathcal{L}(V,V')$. We introduce the shorthand notation
\begin{equation*}
\partial_q A(d) := \partial_q A(q)[d] = \sum_{p=1}^{N_Q} d_p A_p
\text{ and }
\partial_u A(q)\tilde{u} \coloneqq \partial_u (A(q)u)[\tilde{u}] = A(q)\tilde{u}.
\end{equation*}
The adjoint operators $(\partial_q A(d))'$ and $(\partial_u A(q))' \in \mathcal{L}(V,V')$ are defined analogously to~\eqref{eq:adjoint_A}. Thus, if $u$ and $p$ denote the primal and adjoint state corresponding to $q$, the gradient at $q$ in the direction $e \in \Q$ is given as
\begin{equation}
    \langle \nabla J(q), e \rangle_\Q =  
    \int_{0}^T\langle (\partial_qA(e))'u(t), p(t) \rangle_{V',V}\,\text{d}t.
\end{equation}
Let ${u^\iteridx = \mathcal{S}(q^\iteridx) \in W(0,T)}$ be the state corresponding to the current iterate ${q^\iteridx \in \Qad}$. The iterates are updated according to
\begin{equation}
    q^{(i+1)} = q^\iteridx + d^\iteridx \in \Qad,
\end{equation}
where $d^\iteridx$ solves the strictly convex optimization problem \eqref{eq:IRGNMscheme_minimize} for some regularization parameter $\alpha^\iteridx > 0$. The update $d^\iteridx$ is optimal if and only if
\begin{equation*}
    d^\iteridx = P_{\Qad} \left[q^\iteridx + d^\iteridx - \nabla_d \tilde{J}(d^\iteridx; q^\iteridx, \alpha^\iteridx) \right]- q^\iteridx,
\end{equation*}
where $P_{\Qad}:\Q \rightarrow\Qad$ denotes the projection onto the non-empty, closed convex set $\Qad \subset \Q$, defined by
\begin{equation}
    \label{eq:projection_operator}
    P_{\Qad}(x)\coloneqq\argmin \|x - y\|_\Q \text{ s.t. } y \in \Qad,
\end{equation}
cf.~\cite{nocedal2006numerical}. For solving the linearized problem~\eqref{eq:IRGNMscheme_minimize}, efficient evaluation to the gradient of the linearized objective is additionally required. To this end, we introduce the linearized system for the state derivative $\ulin \coloneqq \partial_q u(q)[d] \in L^2(0,T;V)$ in direction $d \in \Q$ and the associated adjoint state $\plin \in L^2(0,T;V)$:
\begin{equation}
\label{eq:linearized_system}
\begin{aligned}
\langle \rho \partial_t^2 \ulin(t) + \partial_u A(q)\ulin(t), v \rangle_{V',V}
+ \langle \partial_q A(d)u(t), v \rangle_{V',V}
&= 0, \\
\langle \rho \partial_t^2 \plin(t) + (\partial_u A(q))'\plin(t), v \rangle_{V',V}
+ \langle \C\big(u(t) + \ulin(t) - y^\delta(t)\big), \C v \rangle_{C}
&= 0,
\end{aligned}
\end{equation}
for all $v \in V$ and f.a.a. $t \in (0,T)$, where $u = \mathcal{S}(q)$. The initial and terminal conditions are given by
\begin{equation*}
\ulin(0) = \partial_t \ulin(0) = 0, 
\qquad
\plin(T) = \partial_t \plin(T) = 0 
\quad \text{in } H.
\end{equation*}
Under Assumption~\ref{ass:ass_to_stored_energy}, this system admits a unique solution satisfying
\begin{equation*}
\ulin \in W(0,T), 
\qquad
\plin \in H^1(0,T;H)\cap L^2(0,T;V),
\qquad
\partial_t^2 \plin \in L^2(0,T;V').
\end{equation*}
The gradient of the linearized objective $\tilde{J}(d; q, \alpha)$ in direction $e \in \Q$ is given by
\begin{equation}
\label{eq:lin_grad}
\langle \nabla_d \tilde{J}(d; q, \alpha), e \rangle_{\Q}
=
\int_0^T 
\langle (\partial_q A(e))'u(t), \plin(t) \rangle_{V',V}
\,\mathrm{d}t
+ \alpha \langle q + d - q_\circ, e \rangle_{\Q}.
\end{equation}

\section{Numerical models}
\label{sec:num_models}

Our goal is not to apply the \gls{IRGNM} directly with a standard \gls{FOM}, such as the one obtained by the finite element method, since solving such models is typically computationally too expensive. This is particularly critical in the multi-query setting of \gls{IRGNM}, where the (linearized) state and adjoint equations must be solved repeatedly for many parameter update directions. Instead, we employ the \emph{\acrlong{TR} \gls{IRGNM}} (\acrshort{TR}-\gls{IRGNM}), introduced in \cite{kartmann2025adaptivereducedbasistrust}. Its main advantage over the classical \gls{FOM}-\gls{IRGNM} is the use of iteratively constructed \gls{RB} models, which replace the \gls{FOM} in most computations, while the latter is only used to enrich the reduced models.

Under Assumption~\ref{ass:ass_to_stored_energy}, the methodology and definitions remain largely analogous to those in \cite{kartmann2025adaptivereducedbasistrust}, since the equations are linear both in the state and in the parameter. The main difference is that we consider hyperbolic equations, which requires adapting both the numerical solution procedure for discretized PDEs and the corresponding \emph{a posteriori} error estimators.

\subsection{Finite Element model}
\label{ssec:discretization}
The function spaces introduced in Section~\ref{ssec:realization_IRGNM} are approximated by the finite-dimensional subspaces
\begin{equation*}
    V_h \coloneqq \mathrm{span}\{\varphi_1, \dots, \varphi_{N_V}\} \subset V,
    \quad
    C_h \coloneqq \mathrm{span}\{c_1, \dots, c_{N_C}\} \subset C,
\end{equation*}
with respective dimensions $N_V, N_C \in \mathbb{N}$. We identify $H_h \coloneqq V_h \subset H$. Similarly, we set 
\begin{equation*}
Q_h \coloneqq \Q,\, N_Q \coloneqq N_\Q \text{ and } \qad \coloneqq \Qad
\end{equation*}
and denote the basis of $Q_h$ by $\{q_1, \dots, q_{N_Q}\}$. The corresponding discrete operators are defined by
\begin{align*}
    \bm{M}_Q &\coloneqq \left(\langle q_i,q_j \rangle_{Q}\right)_{i,j \in \{1, \dots, N_Q\}},&  \bm{M}_V &\coloneqq \left(\langle \varphi_i,\varphi_j \rangle_{V}\right)_{i,j \in \{1, \dots, N_V\}},\\
    \bm{M}_H &\coloneqq \left(\langle \varphi_i,\varphi_j \rangle_{H}\right)_{i,j \in \{1, \dots, N_V\}},& \bm{M}_C &\coloneqq \left(\langle c_i,c_j \rangle_{C}\right)_{i,j \in \{1, \dots, N_C\}},\\
    \bm{L}_h(t)&\coloneqq\left(\langle l(t),\varphi_j\rangle_{V',V}\right)_{j\in \{1,\dots,N_V\}}&\text{for}&~t\in(0,T),\\
    \bm{A}_h(q_h)&\coloneqq \left(\langle A(q_h)\varphi_j,\varphi_i \rangle_{V',V}\right)_{i,j \in \{1, \dots, N_V\}}&\text{for}&~q_h\in Q_h. 
\end{align*}
We also introduce an equidistant temporal grid with uniform time steps
\begin{equation*}
    \{0 \eqqcolon t^0 < t^1 < \dots < t^K \coloneqq T\} \subset [0,T], \quad \Delta t \coloneqq t^{k} - t^{k-1}
\end{equation*}
for all $k \in \mathbb{K} \coloneqq \{1, \dots, K\}$. We write $v = [v^1, \dots, v^K] \in V_h^K$ for a discrete trajectory with values in $V_h$, where $V_h^K$ denotes the product space of $V_h$. In the following, we identify $v$ with its DoF-matrix $\bm{v}=[\bm{v}^1,\dots,\bm{v}^K]\in \R^{N_V\times K}$ where $v^k = \sum_{i = 1}^{N_V} \bm{v}^k_{i}\varphi_i$. We equip these spaces with the discrete inner products
\begin{equation}
    \label{eq:disc_bochner_prods}
    \langle v, w \rangle_{V_h}\coloneqq \bm{v}^\top \bm{M}_V \bm{w} \text{ and } \langle v, w \rangle_{V_h^K}\coloneqq \Delta t \sum_{k = 1}^K \langle v^k, w^k \rangle_{V_h}
\end{equation}
respectively. Analogous definitions apply to the spaces $H_h$, $Q_h$ and $C_h$. 

We employ a generalized trapezoidal rule for the temporal discretization. For a fixed parameter $\zeta \in [0,1]$, the coefficient matrices $\bm{u}_h, \dot{\bm{u}}_h \in \mathbb{R}^{N_V \times K}$ associated with the discrete states $u_h, \dot{u}_h \in V_h^K$ for $q_h \in Q_h$ are obtained as the unique solution of the state system

\begin{equation*}
    \begin{aligned}
        \frac{1}{\Delta t} \rho\bm{M}_H (\dot{\bm{u}}_h^k - \dot{\bm{u}}_h^{k-1}) + \bm{A}_h(q_h)(\zeta \bm{u}_h^k + (1-\zeta) \bm{u}_h^{k-1}) &= \zeta\bm{L}_h^k + (1-\zeta)\bm{L}_h^{k-1}, \\
        \frac{1}{\Delta t} \rho\bm{M}_H (\bm{u}_h^k - \bm{u}_h^{k-1}) & = \zeta \rho\bm{M}_H \dot{\bm{u}}_h^k + (1 - \zeta) \rho\bm{M}_H \dot{\bm{u}}_h^{k-1},
    \end{aligned}
\end{equation*}
for all $k \in \mathbb{K}$, subject to the initial conditions
\begin{equation*}
\bm{u}_h^0 \coloneqq \left((u_0,\varphi_j)_{V}\right)_{j\in \{1,\dots,N_V\}} \text{ and }
\dot{\bm{u}}_h^0 \coloneqq \left((u_1,\varphi_j)_{V}\right)_{j\in \{1,\dots,N_V\}},
\end{equation*}
and $\bm{L}_h^k\coloneqq \bm{L}_h(t^k)$. Reformulation yields
\begin{align}
\label{eq:FOM_primal_PDE_system}
    \begin{aligned}
        \rho\bm{M}_H \dot{\bm{u}}_h^k &= \rho\bm{M}_H \dot{\bm{u}}_h^{k-1} - \Delta t \bm{A}_h(q_h)\bm{u}^{k - 1 + \zeta}_h + \Delta t  \bm{R}^{k-1+\zeta}_{1,h}, \\
        \bm{S}_h(q_h, \zeta) \bm{u}_h^k &= \bm{S}_h(q_h,\zeta-1) \bm{u}_h^{k-1} + \rho\Delta t \bm{M}_H \dot{\bm{u}}_h^{k-1} + \zeta \Delta t  \bm{R}^{k-1+\zeta}_{1,h},
    \end{aligned}
\end{align}
for all $k \in \mathbb{K}$ with
\begin{align*}
    \bm{u}^{k - 1 + \zeta}_h & \coloneqq \zeta \bm{u}_h^k + (1-\zeta) \bm{u}_h^{k-1}, \\
    \bm{S}_h(q_h, \chi) & \coloneqq \rho\bm{M}_H + (\Delta t)^2 \zeta \chi \bm{A}_h(q_h), \\
    \bm{R}^{k- 1 + \zeta}_{1,h} & \coloneqq \zeta\bm{L}_h^k + (1-\zeta)\bm{L}_h^{k-1}.
\end{align*}

\begin{remark}
For implementation purposes, we solve the reformulated system for the intermediate state $\bm{u}_h^{k-1+\zeta}$ rather than $\bm{u}_h^k$. The state $\bm{u}_h^k$ is then recovered explicitly via
\begin{equation*}
\bm{u}_h^k = \frac{1}{\zeta}\bigl(\bm{u}_h^{k-1+\zeta} - (1-\zeta)\bm{u}_h^{k-1}\bigr).    
\end{equation*}
\end{remark}
The discretized solution operator $\Sol_h : \qad \rightarrow V_h^K$ is defined by assigning to each $q_h \in Q_h$ the solution to \eqref{eq:FOM_primal_PDE_system}. We define the discrete observation operator $\C_h: V_h \rightarrow C_h$ by assigning to $v_h \in V_h$ the vector $w_h = \C_h v_h$ defined as
\begin{align*}
    & \bm{w}_h \coloneqq \bm M_C^{-1}
	\left(\langle \C v_h,c_j\rangle_{C}\right)_{j\in \{1,\dots,N_C\}}.
\end{align*} 
The discrete forward operator is given as $\F_h := \C_h \circ \Sol_h : \qadK \rightarrow C_h^K$, with $\F_h(q_h^k) = \C_hu^k_h$ for all $k \in \mathbb{K}$. Analogously, we define $y_h^\delta \in C_h^K$ as
\begin{align*}
    & \bm{y}^{\delta,k}_h\coloneqq \bm M_C^{-1}
	\left(\langle y^\delta(t^k),c_j\rangle_{C}\right)_{j\in \{1,\dots,N_C\}}
	\text{ for }k\in \mathbb K.
\end{align*}
We also introduce the discrete (bi)linear forms
\begin{align*}
     &\bm{C}_h \coloneqq \bm \left(\langle \C_h\varphi_j, \C_h\varphi_i \rangle_{C_h}\right)_{i,j \in \{1, \dots, N_V\}},\ {\bm{Cy}^{\delta,k}_h}\coloneqq \bm \left(\langle \C_h\varphi_j, y_{h}^{\delta,k}\rangle_{C_h}\right)_{1\le j \le N_V} \text{ for }k\in \mathbb K,\\
     &c_1(\bm{y}_h^{\delta}) \coloneqq {\Delta t} \sum_{k = 1}^K\tfrac{1}{2}\big({\bm{y}_h^{\delta,k}}^\top \bm M_C {\bm{y}_h^{\delta,k}} \big).
\end{align*}
Then, for $\Sol_h(q_h)=u_h$ with $q_h\in Q_h$,  the discrete objective functional is given by
\begin{align}
    \label{eq:J_h}
    \begin{aligned}
        J_h(q_h) &\coloneqq \frac{\Delta t}{2} \sum_{k = 1}^K \|\F_h(q_h)^k - y_{h}^{\delta,k} \|^2_{C_h},\\
        &\hspace{0.5mm}=c_1(\bm{y}_h^{\delta})+{\Delta t}  \sum_{k = 1}^K  \tfrac{1}{2}\big({\bm{u}_h^{k\top}}\bm C_h \bm{u}_h^k \big)-\big({\bm{u}_h^{k\top}} \bm{Cy}^{\delta,k}_h\big).
    \end{aligned}
\end{align}
Therefore, the discrete gradient $\nabla J_h(q_h) \in Q_h$ of \eqref{eq:J_h} is component-wise given by
\begin{equation*}
    \bm{\nabla J}_h(q_h) = \Delta t \sum_{k=1}^K \bm M_Q^{-1}\bm B_h(u_h^k)^\top \bm{p}^k_h
\end{equation*}
where $\bm{p}_h \in \mathbb{R}^{N_V \times K}$ denote the DoF-Matrix of the discrete adjoint and
\begin{equation*}
\bm{B}_h(u_h^k)\coloneqq \left( \langle A_j u_h^k, \varphi_i \rangle_{V',V} \right)_{i \in \{1, \dots, N_V\}, j \in \{1, \dots, N_Q\}} \in \mathbb{R}^{N_V \times N_Q},
\end{equation*}
where $A_j$ are the operators introduced in \eqref{eq:A_affine}. The discrete adjoint system can be solved analogously to the state system. The vectors $\bm{p}_h$ and $\dot{\bm{p}}_h \in \mathbb{R}^{N_V \times K}$ are the unique solutions of
\begin{equation}
\label{eq:FOM_adjoint_PDE_system}
    \begin{aligned}
        \rho\bm{M}_H \dot{\bm{p}}_h^k &= \rho\bm{M}_H \dot{\bm{p}}_h^{k+1} - \Delta t \bm{A}^\top_h(q_h)\bm{p}^{k + 1 - \zeta}_h + \Delta t\bm{R}^{k+1-\zeta}_{2,h}(\bm{u}_h) \\
        \bm{S}^\top_h(q_h, \zeta) \bm{p}_h^k &= \bm{S}^\top_h(q_h,\zeta-1) \bm{p}_h^{k+1} + \rho\Delta t
        \bm{M}_H \dot{\bm{p}}_h^{k+1} + \zeta \Delta t \bm{R}^{k+1-\zeta}_{2,h}(\bm{u}_h).
    \end{aligned}
\end{equation}
for all $k \in \{0, \dots, K - 1\}$ with terminal conditions $\bm{p}_h^{K}= \dot{\bm{p}}_h^{K} = \bm{0}$ and 
\begin{align*}
    \bm{p}^{k + 1 -\zeta}_h & \coloneqq \zeta \bm{p}_h^k + (1-\zeta) \bm{p}_h^{k+1}, \\
    \bm{R}^{k+1-\zeta}_{2,h}(\bm{u}_h) & \coloneqq \zeta(\bm{Cy}^{\delta,k}_h - \bm C_h \bm{u}_h^k) + (1- \zeta)(\bm{Cy}^{\delta,k+1}_h - \bm C_h \bm{u}_h^{k+1}).
\end{align*}
Similarly, for given parameters $q_h \in Q_h$ and direction $d_h \in Q_h$, the coefficients of the linearized state and adjoint functions $\tilde{\bm{u}}_h, \dot{\tilde{\bm{u}}}_h, \tilde{\bm{p}}_h, \dot{\tilde{\bm{p}}}_h \in \mathbb{R}^{N_V \times K}$ are calculated as solutions to the system
\begin{equation}
\label{eq:FOM_lin_primal_PDE_system}
    \begin{aligned}
        \rho\bm{M}_H \dot{\tilde{\bm{u}}}_h^k &= \rho\bm{M}_H \dot{\tilde{\bm{u}}}_h^{k-1} - \Delta t \bm{A}_h(q_h)\tilde{\bm{u}}^{k - 1 + \zeta}_h + \Delta t \bm{R}^{k+1-\zeta}_{3,h}(\bm{u}_h, \bm{d}_h), \\
        \bm{S}_h(q_h, \zeta) \tilde{\bm{u}}_h^k &= \bm{S}_h(q_h,\zeta-1) \tilde{\bm{u}}_h^{k-1} + \rho\Delta t \bm{M}_H \dot{\tilde{\bm{u}}}_h^{k-1} + \zeta \Delta t \bm{R}^{k+1-\zeta}_{3,h}(\bm{u}_h,\bm{d}_h),
    \end{aligned}
\end{equation}
for all $k \in \mathbb{K}$ and
\begin{equation}
    \label{eq:FOM_lin_adjoint_PDE_system}
    \begin{aligned}
        \rho\bm{M}_H \dot{\tilde{\bm{p}}}_h^k &= \rho\bm{M}_H \dot{\tilde{\bm{p}}}_h^{k+1} - \Delta t \bm{A}^\top_h(q_h)\tilde{\bm{p}}^{k + 1 - \zeta}_h  + \Delta t\bm{R}^{k+1-\zeta}_{4,h}(\bm{u}_h), \\
        \bm{S}^\top_h(q_h, \zeta) \tilde{\bm{p}}_h^k &= \bm{S}^\top_h(q_h,\zeta-1) \tilde{\bm{p}}_h^{k+1} + \rho\Delta t \bm{M}_H \dot{\tilde{\bm{p}}}_h^{k+1} + \zeta \Delta t \bm{R}^{k+1-\zeta}_{4,h}(\bm{u}_h).
    \end{aligned}
\end{equation}
for all $k \in \{0, \dots, K - 1\}$, subject to the initial and terminal conditions 
\begin{equation*}
\tilde{\bm{u}}_h^{0}= \dot{\tilde{\bm{u}}}_h^{0} = \tilde{\bm{p}}_h^{K}= \dot{\tilde{\bm{p}}}_h^{K} = \bm{0},
\end{equation*}
where we define
\begin{align*}
    \tilde{\bm{u}}^{k - 1 +\zeta}_h & \coloneqq \zeta \tilde{\bm{u}}_h^k + (1-\zeta) \tilde{\bm{u}}_h^{k-1}, \\
    \tilde{\bm{p}}^{k + 1- \zeta}_h & \coloneqq \zeta \tilde{\bm{p}}_h^k + (1-\zeta) \tilde{\bm{p}}_h^{k+1}, \\
    \bm{R}^{k-1+\zeta}_{3,h}(\bm{u}_h,\bm{d}_h) & \coloneqq - \zeta \bm{B}_h(u_h^k) \bm{d}_h - (1-\zeta)\bm{B}_h(u_h^{k-1}) \bm{d}_h, \\
    \bm{R}^{k+1-\zeta}_{4,h}(\bm{u}_h) & \coloneqq \zeta(\bm{C}\tilde{\bm{y}}_h^{\delta,k} - \bm C_h \tilde{\bm{u}}_h^k) +(1-\zeta)(\bm{C}\tilde{\bm{y}}_h^{\delta,k+1} - \bm C_h \tilde{\bm{u}}_h^{k+1}).
\end{align*}
with $\bm{C}\tilde{\bm{y}}^{\delta,k}_h \coloneqq {\bm{Cy}^{\delta,k}_h}-\bm{C}_h \bm{u}_h^{k}$. The regularized linearization of $J_h(q_h)$ for $\alpha > 0$ is given by 
\begin{align*}
    \widetilde{J}_h(d_h; q_h, \alpha)  \coloneqq\, &  c_2(\tilde{\bm{y}}^{\delta}_h)+\Delta t \sum_{k=1}^K \left[
    \tfrac{1}{2} \tilde{\bm{u}}_h^{k\top} \bm{C}_h \tilde{\bm{u}}_h^k 
    - \tilde{\bm{u}}_h^{k\top} \bm{C}\tilde{\bm{y}}^{\delta,k}_h
    \right] \\
    & + \alpha c_3(\tilde{\bm{q}}_{\circ,h}) + \alpha\left[
    \tfrac{1}{2} \bm{d}_h^{\top} \bm{M}_Q \bm{d}_h 
    - \bm{d}_h^{\top} \bm{M}_Q \tilde{\bm{q}}_{\circ,h}  \right],
\end{align*}
where $\tilde{\bm{u}}_h$ is the linearized state associated to $q_h$ and $d_h$ and the constant terms are defined by ${{\tilde{\bm{q}}}_{\circ,h} \coloneqq {\bm{q}_{\circ,h}} - {\bm{q}_{h}}}$, 
\begin{align*}
    &c_2(\tilde{\bm{y}}^{\delta}_h) \coloneqq c_1(\bm{y}_h^{\delta})+ {\Delta t} \sum_{k = 1}^K\tfrac{1}{2}\big({\bm{u}_h^{k\top}} \bm C_h {\bm{u}_h^{k}} \big)-\big({\bm{u}_h^{k\top}} \bm{Cy}^{\delta,k}_h\big),\\
    &c_3({\tilde{\bm{q}}}_{\circ,h})\coloneqq \tfrac{1}{2} \tilde{\bm{q}}_{\circ,h}^{\top} \bm{M}_Q \tilde{\bm{q}}_{\circ,h}.
\end{align*}
The gradient of the discrete linearized objective function is then given by
\begin{equation*}
    \bm{\nabla_d} \widetilde{\bm{J}}_{h}(d_h; q_h, \alpha) = \Delta t \sum_{k=1}^K \bm{M}_Q^{-1} \bm{B}_h(u_h^k)^\top \tilde{\bm{p}}_h^k + \alpha (\bm{d}_h - \tilde{\bm{q}}_{\circ,h}).
\end{equation*}
The FOM-IRGNM solves the discretized optimization problem corresponding to \eqref{eq:minJ}, defined as
\begin{equation}
	\label{eq:minJ_FOM}
	\tag{$\mathbf{IP}_{h}$}
	\min J_h(q_h) \quad \text{s.t.}\quad q_h \in \qad.
\end{equation}
The procedure is completely analogous to that described in Section~\ref{subsec:IRGNM}. The updates $d^\iteridx_h \in Q_h$ are computed by iteratively solving  
\begin{align}
    \label{prob:lin_prob_FOM}
    \tag{$\widetilde{\mathbf{IP}}_{\alpha,h}$}
    d^\iteridx_h\coloneqq\argmin \widetilde{J}_h(d_h; q^\iteridx_h, \alpha^{\iteridx})\quad \text{s.t.}\quad q^\iteridx_h + d_h \in  \qad,
\end{align}
until 
\begin{equation}
	\label{eq:choice_of_alpha_FOM}
	\theta J_h(q_h^\iteridx)\leq 2  \widetilde{J}_h(d^\iteridx_h; q^\iteridx_h, 0)\leq \Theta J_h(q_h^\iteridx),
\end{equation}
is satisfied for the selected $\alpha^\iteridx$. The update is then accepted, and the next iterate is defined by ${q_h^{(i+1)} := q_h^\iteridx + d^\iteridx_h}$.

\subsection{Reduced-Basis model}
\label{ssec:modelreduction}
As discussed above, we aim to replace the FOM by a reduced-order model (ROM) obtained by projecting the system equations onto reduced-basis spaces, which is a natural choice for model order reduction in the multi-query setting of the IRGNM. Following \cite{kartmann2025adaptivereducedbasistrust, kartmann_adaptive_2023}, we employ a combined reduction of the parameter and state spaces, since the high dimensionality of the parameter space poses significant challenges to established methods, which typically only a reduce in the state space. This combined approach is particularly well suited for achieving an efficient offline–online decomposition.

We therefore introduce the reduced state and parameter spaces
\begin{equation*}
V_r=\text{span}(\tilde{\varphi}_1,\dots,\tilde{\varphi}_{n_V}) \subset V_h,\quad Q_r=\text{span}(\tilde{q}_1,\dots,\tilde{q}_{n_Q}) \subset Q_h,     
\end{equation*}
with dimensions $n_Q \ll N_Q$ and $n_V \ll N_V$ given. Analogously to the FOM, we identify $H_r \coloneqq V_r$ and denote the basis matrices as $\bm \Psi_Q \in \R^{N_Q\times n_Q}$, $\bm \Psi_V \in \R^{N_V\times n_V}$ and assume that the bases are orthonormal with respect to the corresponding inner products, i.e.,
\begin{equation*}
\bm \Psi_V^\top \bm M_V \bm \Psi_V = \bm I_{n_V}
\text{ and } 
\bm \Psi_Q^\top \bm M_Q \bm \Psi_Q= \bm I_{n_Q}.     
\end{equation*}
The corresponding reduced operators are given by
\begin{align*}
    &\bm{M}_{Q,r} \coloneqq \bm I_{n_Q}, \quad \bm{M}_{V,r} \coloneqq \bm I_{n_V},\quad \bm{M}_{H,r} \coloneqq \bm \Psi_V^\top \bm{M}_{H} \bm \Psi_V.
\end{align*}
For $q_r \in Q_r$, which we identify with its coefficient vector $\bm q_r \in \R^{n_Q}$, the reduced operator $\bm{A}_r(q_r)$ admits an affine decomposition
\begin{equation*}
    \bm{A}_r(q_r) = \bm{A}_{r,0} + \sum_{j=1}^{n_Q} \bm q_{r,j}\, \bm{A}_{r,j},
\end{equation*}
where
\begin{equation*}
    \bm{A}_{r,0} \coloneqq \bm \Psi_V^\top \bm{A}_{h,0}\bm \Psi_V
    \text{ and }
    \bm{A}_{r,j} \coloneqq \bm \Psi_V^\top \bm{A}_h(\tilde{q}_j)\bm \Psi_V - \bm{A}_{r,0}.
\end{equation*}
The matrices $\bm{A}_{r,0}$ and $\bm{A}_{r,j}$ can be precomputed during the construction of the reduced-basis model (offline phase) and subsequently reused during the parameter inference (online phase). Since the reduced parameter space $Q_r$ is typically of much lower dimension than $Q_h$, evaluating this affine decomposition is significantly cheaper than assembling the full-order matrix $\bm{A}_h(q_r)$.

Given a reduced parameter $q_r \in Q_r$, the RB approximation $\bm u_r$ of the state is obtained by solving the primal system \eqref{eq:ROM_primal_PDE_system}, projected to the reduced state space $V_r$. The corresponding system is given by
\begin{align}
\label{eq:ROM_primal_PDE_system}
    \begin{aligned}
        \rho\bm{M}_{H,r} \dot{\bm{u}}_r^k &= \rho\bm{M}_{H,r} \dot{\bm{u}}_r^{k-1} - \Delta t \bm{A}_r(q_r)\bm{u}_r^{k-\zeta} + \Delta t  \bm{R}^{k-\zeta}_{1,r}, \\
        \bm{S}_r(q_r, \zeta) \bm{u}_r^k &= \bm{S}_r(q_r,\zeta-1) \bm{u}_r^{k-1} + \Delta t \rho\bm{M}_{H,r} \dot{\bm{u}}_r^{k-1} + \zeta \Delta t  \bm{R}^{k-\zeta}_{1,r},
    \end{aligned}
\end{align}
for all $k \in \mathbb{K}$, subject to initial conditions $\bm{u}_r^{0} \coloneqq \bm \Psi_V^\top \bm{u}_h^{0}$ and 
$\dot{\bm{u}}_r^{0} \coloneqq \bm \Psi_V^\top \dot{\bm{u}}_h^{0}$, where
\begin{align*}
    \bm{u}_r^{k-\zeta} & \coloneqq \zeta \bm{u}_r^k + (1-\zeta) \bm{u}_r^{k-1} \\
    \bm{S}_r(q_r, \chi) & \coloneqq \rho\bm{M}_{H,r} + (\Delta t)^2 \zeta \chi \bm{A}_r(q_r) \\
    \bm{R}^{k-\zeta}_{1,r} & 
    \coloneqq \bm \Psi_V^\top \bm{R}^{k-\zeta}_{1,h}.
\end{align*}
The RB approximation of the adjoint $\bm p_r$ is obtained analogously by projecting the adjoint system~\eqref{eq:FOM_adjoint_PDE_system} onto $V_r$. For a direction $d_r \in Q_r$, the RB approximations of the linearized state and adjoint, denoted by $\tilde{\bm u}_r$ and $\tilde{\bm p}_r$, are computed in a similar manner. 

We write $u_r=\Sol_r(q_r)$ and define $\F_r=\C_h\circ\Sol_r : \qadr \rightarrow C_h^K$, where the set of reduced admissible parameters $\qadr \subset Q_r$ is defined as
\begin{equation}
\label{eq:def_reduced_admissiable_set}
    \qadr
    \coloneqq
    \left\{
    \bm q_r \in Q_r
    \;\middle|\;
    \bm{\Psi}_Q \bm q_r \in \qad
    \right\}.
\end{equation}
The reduced objective for $q_r\in Q_r^K$ is thus given by
\begin{align}
    \label{eq:J_r}
    J_r(q_r) \coloneqq 
    & c_1(\bm{y}_h^{\delta}) + {\Delta t}\sum_{k = 1}^K \tfrac{1}{2}\big({\bm{u}_r^{k\top}} \bm{C}_r \bm{u}_r^k \big)-\big({\bm{u}_r^{k\top}} \bm{C}\bm{y^{\delta,k}}_r \big)
\end{align}
where
\begin{align*}
    &\bm{C}_r=\bm\Psi_V^\top \bm C_h \bm\Psi_V,\quad \bm{Cy}^{\delta,k}_r=\bm\Psi_V^\top \bm C\bm{y}_h^{\delta,k} \quad \forall k \in \mathbb{K}.
\end{align*}
The associated discrete gradient is given by 
\begin{equation*}
    \bm{\nabla J}_r(q_r) = \Delta t \sum_{k=1}^K \bm B_r(u_r^k)^\top \bm{p}^k_r,\
\text{where }
    \bm{B}_r(u_r^k) 
    \coloneqq 
    \bm{\Psi}_V^\top \bm{B}_h(u_r^k)\bm{\Psi}_Q 
    \in \mathbb{R}^{n_V \times n_Q}.
\end{equation*}
Note that $\Sol_r$ is well-defined and that the online evaluation of $J_r(q_r)$ does not require any full-order computations. The reduced optimization problem can thus be stated as
\begin{equation}
\label{prob:prob_ROM}
	\tag{$\textbf{IP}_{r}$} 
	\min{J_r(q_r)}\text{ s.t. } q_r \in \qadr.
\end{equation}
We further define ${\tilde{\bm{ q}}}_{r,0}=\bm \Psi_Q^\top \bm M_Q\tilde{\bm{ q}}_{h,0}$, $\bm{C}\tilde{\bm{y}}^{\delta,k}_r={\bm{Cy}^{\delta,k}_r}-\bm{C}_r \bm{u}_r^{k}$ and the constant terms
\begin{align*}
    &c_2(\tilde{\bm{y}}^{\delta}_r) \coloneqq c_1(\bm{y}_h^{\delta})+ {\Delta t} \sum_{k = 1}^K\tfrac{1}{2}\big({\bm{u}_r^{k\top}} \bm C_r {\bm{u}_r^{k}} \big)-\big({\bm{u}_r^{k\top}} \bm{Cy}^{\delta,k}_r\big),\
    c_3({\tilde{\bm{q}}}_{\circ,r})\coloneqq \tfrac{1}{2} \tilde{\bm{q}}_{\circ,r}^{\top} \tilde{\bm{q}}_{\circ,r}.
\end{align*}
The linearized reduced objective $\widetilde{J}_r$ is defined as
\begin{align*}
    \widetilde{J}_r(d_r; q_r, \alpha)  \coloneqq &  c_2(\tilde{\bm{y}}^{\delta}_r)+\Delta t \sum_{k=1}^K \left[
    \tfrac{1}{2} \tilde{\bm{u}}_r^{k\top} \bm{C}_r\tilde{\bm{u}}_r^k 
    - \tilde{\bm{u}}_r^{k\top} \bm{C}\tilde{\bm{y}}^{\delta,k}_r
    \right] \\
    & +  \alpha c_3({\tilde{\bm{q}}}_{\circ,r})+ \alpha \left[
    \tfrac{1}{2} \bm{d}_r^{\top} \bm{d}_r
    - \bm{d}_r^{\top} \tilde{\bm{q}}_{\circ,r}  \right],
\end{align*}
where we have used $\bm M_{Q,r}=\bm I_{n_Q}$. The reduced linearized subproblem is thus given by
\begin{align}
\label{prob:lin_prob_rom}
	\tag{$\widetilde{\textbf{IP}}_{\alpha, r}$}
    d^\iteridx_r	:= \argmin \tilde{J}_r(d_r; q^\iteridx_r, \alpha^\iteridx) \text{ s.t. } q^\iteridx_r + d_r \in \qadr.
\end{align}
The discrepancy principle for $0 < \theta < \Theta < 2$ is given by
\begin{equation}
	\label{eq:choice_of_alpha_ROM}
	\theta J_r(q_r^\iteridx)\leq 2\widetilde{J}_r(d^\iteridx_r; q^\iteridx_r, 0) \leq \Theta J_r(q_r^\iteridx).
\end{equation}
The coefficients of $\widetilde{J}_r$ at $q_r \in \qadr$ are computed as
\begin{equation*}
    \bm{\nabla_d} \widetilde{\bm{J}}_r(d_r; q_r, \alpha) = \Delta t \sum_{k=1}^K \bm B_r(u_r^k)^\top \tilde{\bm{p}}^k_r + \alpha(\bm{d}_r-\tilde {\bm{q}}_{\circ, r}).
\end{equation*} 

\subsubsection{Error estimation}
\label{ssec:errorest}
In this section, we derive an a posteriori error estimator for the reduced objective $J_r$. More precisely, we seek a estimator
\begin{equation*}
\objerrorestsymb : Q_r \times \mathbb{R}_{\geq 0} \to \mathbb{R}_{\geq 0}
\end{equation*}
such that, for all $q_r \in Q_r$ holds
\begin{equation}
\label{eq:objective_error_est}
|J_h(q_r) - J_r(q_r)| \leq \objerrorest{q_r}{J_r(q_r)}.
\end{equation}
The key step is the derivation of an error estimator for the reduced state $u_r$. To derive the state estimator, we adapt standard techniques from reduced-basis theory. To this end, for $k \in \mathbb{K}$, $q \in Q_h$, and $u, \dot{u} \in V_h^K$, we introduce the \textit{primal residual operator} $r^k_{\mathrm{pr}}(u, \dot{u}; q) \in V_h'$ defined by
\begin{equation}
\label{eq:primal_residual}
\bm{r}^k_{\mathrm{pr}}(u, \dot{u}; q)
=
\bm{R}^{k-1+\zeta}_{1,h}
- \bm{A}_h(q)\bm{u}^{k-1+\zeta}
- \frac{1}{\Delta t}\rho\bm{M}_H\bigl(\dot{\bm{u}}^k - \dot{\bm{u}}^{k-1}\bigr)
\in \mathbb{R}^{N_V}.
\end{equation}
Based on this residual formulation, an a posteriori error estimate in the discrete energy norm can be derived. The proof follows closely the arguments of~\cite{glas2020reduced,bernardi2005time} and is given in detail in Appendix~\ref{ssec:proof_state_error_est}.
\begin{theorem}
\label{thm:state_apost_error_est}
Let $(u_h^k, \dot{u}_h^k)$ and $(u_r^k, \dot{u}_r^k)$, for $k \in \mathbb{K}$, denote the solutions to~\eqref{eq:FOM_primal_PDE_system} and~\eqref{eq:ROM_primal_PDE_system}, respectively, for $q_r \in \qadr$ and $\zeta \geq \tfrac{1}{2}$. Define the errors
\begin{equation*}
\errprim{k}{} \coloneqq u_h^k - u_r^k \in V_h,
\qquad
\errprimdot{k}{} \coloneqq \dot{u}_h^k - \dot{u}_r^k \in H_h.
\end{equation*}
Then the following a posteriori error estimate holds:
\begin{equation}
\label{eq:state_apost_error_est}
\left(\sum_{k = 1}^K \|\errprim{k}{}\|^2_{E} \right)^{\frac{1}{2}} \leq \stateerrorest{q_r},
\end{equation}
where the energy norm is defined by
\begin{equation*}
\|\errprim{k}{}\|^2_{E}
=
(\errprimdot{k}{\bm})^\top \bm{M}_h \errprimdot{k}{\bm}
+
(\errprim{k}{\bm})^\top \bm{A}_h(q_r) \errprim{k}{\bm}.
\end{equation*}
The right-hand side is given by
\begin{equation*}
\stateerrorest{q_r}
=
2 \left(
\sum_{k = 1}^K
\left(
\Delta t \sum_{k' = 1}^k
\|r^{k'}_{\mathrm{pr}}(u_r, \dot{u}_r; q_r)\|_{H'_h}
\right)^{2}
\right)^{\frac{1}{2}}.
\end{equation*}
\end{theorem}
Using Theorem~\ref{thm:state_apost_error_est}, an error estimator in the sense of~\eqref{eq:objective_error_est} can be derived directly by elementary calculations, cf.~\cite{qian2017certified,klein2026multi}. The proof is provided in Appendix~\ref{ssec:proof_obj_error_est}.
\begin{theorem}
\label{thm:J_apost_error_est}
Let the assumptions of Theorem~\ref{thm:state_apost_error_est} hold, and let $J_h(q_r)$ and $J_r(q_r)$ denote the objective functionals corresponding to $u_h$ and $u_r$, respectively. Then the estimate~\eqref{eq:objective_error_est} holds with
\begin{equation}
\label{eq:J_apost_error_est}
\objerrorest{q_r}{J_r(q_r)}
:=
\frac{\|\mathcal C\|_{\mathcal L(V,C)}^2}{2a_{q_r}}\,
\tildestateerrorest{q_r}^{\,2}
+
\|\mathcal C\|_{\mathcal L(V,C)}
\left(\frac{2J_r(q_r)}{a_{q_r}}\right)^{\frac12}
\tildestateerrorest{q_r},
\end{equation}
where
\begin{equation*}
\tildestateerrorest{q_r}
:=
\left(
\Delta t \sum_{k=1}^K (\errprim{k}{\bm})^\top \bm A_h(q_r)\errprim{k}{\bm}
\right)^{\frac12} \leq \stateerrorest{q_r},
\end{equation*}
and
\begin{equation*}
a_{q_r} = \frac{\tilde{\kappa}}{C_P^2} > 0,
\end{equation*}
with $C_P > 0$ denoting the Poincar\'e constant and $\tilde{\kappa} > 0$ the constant from Assumption~\ref{ass:ass_to_stored_energy}.
\end{theorem}

\section{Trust-Region IRGNM}
\label{sec:TR_IRGNM}
The central idea of the trust-region approach is to replace the full-order optimization problem $\eqref{eq:minJ_FOM}$ by a sequence of surrogate problems. These surrogate problems employ approximations of the full-order objective while restricting the admissible set to (local) trust regions $T^{(i)} \subset \Q$, thereby ensuring the (local) reliability of the surrogate. Such approaches have been successfully applied to a variety of problems, in particular to PDE-constrained optimization~\cite{YueMeerbergen2013, lieberman2010parameter, keil2021non, keil2024relaxed, banholzer2020adaptive, klein2026multi}. We therefore employ the \gls{IRGNM} outlined in Section~\ref{ssec:discretization} to solve a sequence of trust-region subproblems of the form
\begin{equation}
\label{eq:TR_Problem}
q_r^{(i+1)} \coloneqq \argmin_{q_r \in T^{(i)}} J_r^{(i)}(q_r),
\end{equation}
where $i \in \mathbb{I} := \{0, \dotsc, \termidxh - 1\}$, such that for $\termidxh$ the \emph{full-order discrepancy principle} is satisfied for some $\tau > 1$, that is (compare~\eqref{eq:discrepancy_principle}),
\begin{equation*}
J_h(q_r^{\termidxh}) \leq \frac{1}{2}(\tau \delta)^2 \leq J_h(q_r^\iteridx), \quad i \in \mathbb{I}.
\end{equation*}

As surrogates $J_r^{(i)}$, we employ \gls{RB}-\gls{ROM}s, as constructed in Section~\ref{ssec:modelreduction}, associated with iteratively enriched reduced state $V_r^{(i)} \subset V_h$ and reduced parameter spaces $Q_r^{(i)} \subset Q_h$. After each subproblem, the reduced spaces are updated to locally adapt the surrogate to $J_h$. Details on the construction of these spaces are provided in Section~\ref{ssec:enrichement}.

The subproblems are solved by applying the IRGNM, with the (approximate) solution of the previous subproblem as an initial guess and combined with a backtracking procedure to enforce the trust region; see~\cite[Section 4.1]{kartmann2025adaptivereducedbasistrust} for details. For each subproblem, this procedure yields a sequence of parameters converging to a local minimizer of $J_r^\iteridx$ in $T^\iteridx$. We denote the resulting sequence by $\{q_r^{(i,l)}\}_{l = 0}^{L^{(i)}} \subset Q^\iteridx_r$, with $q_r^{(i,0)} = q_r^{(i-1)}$. The \gls{IRGNM} iteration is terminated once the reduced discrepancy principle is satisfied, i.e.,
\begin{equation}
\label{eq:red_disc_princ}
J_r^{(i)}\bigl(q_r^{(i,l)}\bigr) < \frac{1}{2} \tilde{\tau}^2 \bigl(\tilde{\delta}^{(i)}\bigr)^2,
\end{equation}
where $\tilde{\delta}^{(i)} \geq \delta$ and $\tilde{\tau} > 1$, or when the iterates approach the boundary of the trust region sufficiently closely. The choice of the trust region is crucial for the efficiency of the algorithm. Although various definitions are possible, we employ a standard radius-based trust region, i.e.,
\begin{equation*}
T^{(i)} \coloneqq \left\{ q_r \in \qadri \,\bigg|\, \errest(q_r) := \|q_r - q_r^{(i)}\|_{Q_r} \leq \eta^{(i)} \right\},
\end{equation*}
for some radius $0 < \eta\iteridx \leq \eta_{\max}$, where $\qadri$ denotes the admissible parameter set for $Q^\iteridx_r$, cf.~\eqref{eq:def_reduced_admissiable_set}.

\begin{remark}
In contrast to other existing RB trust-region approaches, we do not employ an error-aware trust region, although a theoretically sound a posteriori error estimator for $J^\iteridx_r$ has been derived in Section~\ref{ssec:modelreduction}. While estimator-based trust regions are theoretically appealing, residual-based a posteriori error estimators are known to considerably overestimate the true error. This behavior has been observed for elliptic and parabolic problems and is even more pronounced in the hyperbolic setting considered here~\cite{klein2026multi, kartmann2025adaptivereducedbasistrust, kartmann_adaptive_2023, banholzer2020adaptive}. As a consequence, the resulting trust regions become overly conservative, leading to a substantial increase in both offline and online computational costs. Preliminary numerical experiments confirm this effect and indicate that an estimator-driven trust-region strategy is not practically viable for the problems under consideration. Despite this, note that the radius-based trust region still provides implicit control of the error $\objerrorest{q_r}{J_r(q_r)}$, since both $J_h$ and $J_r$ are continuous. Hence, there exists a constant $C > 0$ such that
\begin{equation*}
\objerrorest{q_r}{J_r(q_r)} \leq C \errest(q_r).
\end{equation*}
\end{remark}

An important caveat in the use of \gls{RB}-based surrogates within a trust-region framework is that the trust region constraint alone does not guarantee a decrease in the full-order objective. In particular, especially in flat regions of the solution manifold, for some $i < \termidxh - 1$,
\begin{equation*}
J_h(q_r^{(i+1)}) > J_h(q_r^\iteridx),
\end{equation*}
since the trust-region only ensures that $J_h$ and $J_r$ agree up to an error $\epsilon_J \geq 0$. Therefore, a solution of $\eqref{eq:TR_Problem}$ returned by the \gls{IRGNM} using the surrogate is treated as a \emph{trial solution} $\qtrial$. It must then be verified that
\begin{equation}
\label{eq:decay_cond}
J_h(\qtrial) < J_h(q_r^\iteridx),
\end{equation}
and the iterate is only accepted if $\eqref{eq:decay_cond}$ is satisfied, i.e., $q_r^{(i+1)} = \qtrial$. This strategy was originally introduced in~\cite{YueMeerbergen2013}. If $\qtrial$ is rejected, i.e., if $\eqref{eq:decay_cond}$ fails, the trust region is shrunk by setting
\begin{equation*}
\eta^{(i+1)} = \beta_3 \eta^\iteridx,
\end{equation*}
with $\beta_3 \in (0,1)$, and the subproblem $\eqref{eq:TR_Problem}$ is solved again starting from $q_r^\iteridx$. If, $\qtrial$ is accepted as $q_r^{(i+1)}$, the trust region is enlarged by setting
\begin{equation*}
\eta^{(i+1)} = \min\left(\eta_{\max}, \beta_3^{-1} \eta^\iteridx\right),
\end{equation*}
provided that
\begin{equation*}
\label{eq:TR_enlarge_cond}
\varrho^\iteridx \coloneqq \frac{J_h\bigl(q_r^\iteridx\bigr) - J_h\bigl(q_r^{(i+1)}\bigr)}
{J_r^\iteridx\bigl(q_r^\iteridx\bigr) - J_r^\iteridx\bigl(q_r^{(i+1)}\bigr)}
\geq \beta_2,
\quad \text{with } \beta_2 \in \left[\tfrac{3}{4}, 1\right),
\end{equation*}
Otherwise, it is kept unchanged.

\subsection{Enrichment Strategy}
\label{ssec:enrichement}

A key component of the TR-IRGNM is the construction of low-dimensional reduced spaces 
$Q_r^{(i)}$ and $V_r^{(i)}$ that maintain high approximation accuracy. In this work, 
we employ an \emph{adaptive enrichment} strategy in which the reduced bases are 
iteratively expanded using full-order snapshots obtained during the optimization. 
This yields nested sequences of spaces:
\begin{equation*}
Q_r^{(0)} \subset Q_r^{(1)} \subset \dots \subset Q_r^{\termidxh} \subset Q_h,
\quad
V_r^{(0)} \subset V_r^{(1)} \subset \dots \subset V_r^{\termidxh} \subset V_h.
\end{equation*}
The choice of enrichment snapshots is guided by the structure of the optimization 
problem. We initialize the parameter reduced space as
\begin{equation*}
Q_r^{(0)} \coloneqq \mathrm{span} \left\{ 
q_r^{(0)}, \, q_{\circ,h}, \, \nabla J_h\bigl(q_r^{(0)}\bigr)
\right\}.
\end{equation*}
During the enrichment process, we additionally include full-order gradients 
$\nabla J_h\bigl(q_r^{(i)}\bigr)$, since they represent directions of steepest descent 
and thus capture the most relevant local sensitivities of the objective functional. For the state space, we incorporate the corresponding primal and adjoint states $u_h\bigl(q_r^{(i)}\bigr)$ and $p_h\bigl(q_r^{(i)}\bigr)$. This construction ensures that, for all $i \in \mathbb{I}$, hold
\begin{equation}
\label{eq:tr_compliance_cond}
q_r^{(i)} \in Q_r^{(i)},
\quad
J_h\bigl(q_r^{(i)}\bigr) = J_r^{(i)}\bigl(q_r^{(i)}\bigr)
\quad 
\text{and} 
\quad 
\nabla J_h\bigl(q_r^{(i)}\bigr) = \nabla  J_r^{(i)}\bigl(q_r^{(i)}\bigr),
\end{equation}
ensuring first-order consistency of the surrogate at $q_r^{(i)}$. In particular, this construction implies ${q_r^{(i)} \in \mathcal{T}^{(i)}}$, hence $\mathcal{T}^{(i)} \neq \emptyset$, guarantying the well-posedness of the trust-region procedure. 

To limit the growth of the reduced state space and avoid the accumulation of irrelevant basis vectors, we apply \gls{POD} to retain the most relevant components from the collected snapshots up to a prescribed tolerance ${\epsPODstate > 0}$, cf.~\cite{Haasdonk2008,Haasdonk13}. In practice, we employ \gls{HaPOD} to enhance computational efficiency~\cite{himpe2018hierarchical}. We refer to~\cite{kartmann2025adaptivereducedbasistrust} for a detailed description of the enrichment procedure. The combination of adaptive enrichment and the trust-region framework yields the TR-IRGNM algorithm, whose main steps are summarized in Algorithm~\ref{algo:TR_IRGNM}.

\begin{algorithm2e}
\DontPrintSemicolon
\caption{TR-IRGNM}\label{algo:TR_IRGNM}
\small
\KwData{Noise level $\delta$, discrepancy parameter $\tau, \tilde{\tau} > 1$, initial guess $q_r^{(0)} \in \Q$, tolerance for enlargement of the radius $\beta_2 \in [3/4,1)$, shrinking factor $\beta_3 \in (0,1)$ and maximal trust region-radius $\eta_{\max} > 0$.\;}
Set $i = 0$ and initialize the RB-ROM using $Q_r^{(0)} \coloneqq \mathrm{span}(\{q_r^{(0)}, q_{\circ,h}, \nabla J_h(q_r^{(0)})\}) $ and $V_r^{(0)} \coloneqq \mathrm{span}(\{u_h(q_r^{(0)}), p_h(q_r^{(0)})\})$.\;
\While{$J_h(q_r^{(i)}) > \frac{1}{2}\tau^2 \delta^2$}{
    Solve \eqref{prob:prob_ROM} for $\qtrial$.\;
    \If{$J_h(\qtrial) \geq J_h(q_r^\iteridx)$}{
        Reject trial step, set $q_r^{(i+1)} \coloneqq q_r^\iteridx$, keep $Q_r^\iteridx$ and $V_r^\iteridx$ and shrink radius $\eta^{(i+1)} \coloneqq \beta_3\eta^\iteridx$.\;
    }
    \Else{
        Accept the trial step as $q_r^{(i+1)} \coloneqq \qtrial$, update $Q_r^\iteridx$ and $V_r^\iteridx$ at $q_r^{(i+1)}$.\;
        \If{$\varrho^\iteridx > \beta_2$}{
            Enlarge radius $\eta^{(i+1)} \coloneqq \min\left(\eta_{\max}, \beta_3^{-1} \eta^\iteridx\right)$.\;
        }
        \Else{
            Keep the trust-region radius unchanged, $\eta^{(i+1)} \coloneqq \eta^\iteridx$.\;
        }
    }    
    $i \leftarrow i + 1$. \label{line:final_line}\;	
    
}
\end{algorithm2e}

\section{Numerical experiments}
\label{sec:num_exper}
\subsection{Setup}
To assess the potential and limitations of the TR-IRGNM method combined with reduced-basis surrogates, we consider the inverse problem of identifying defects on the surface of a thin plate composed of a Cauchy elastic material \cite{lechleiter2017identifying, holzapfel2002nonlinear}. The goal is to reconstruct spatially distributed parameter fields corresponding to the Lamé coefficients from (artificial) measurement data. For our numerical experiments, we consider a thin plate of dimensions $\SI{6.7}{\milli\meter} \times \SI{1}{\meter} \times \SI{1}{\meter}$, which is fixed along the boundaries in $y$- and $z$-directions. The computational domain is defined as $\Omega = [-0.1,0.1] \times [-15,15] \times [-15,15]$. The Dirichlet boundary is set as
\begin{equation*}
    \Gamma \coloneqq \left\{ (x,y,z) \in \partial \Omega \;\middle|\; y = \pm 15 \;\text{or}\; z = \pm 15 \right\}.
\end{equation*}
We aim to simulate the displacement field $u$ over the time interval $[0,T]$, with $T \approx \SI{533}{\micro\second}$. For the implementation, a rescaled time interval ${[0,T] = [0,16]}$ is used. The plate is discretized using $4 \times 60 \times 60$ trilinear finite elements on an equidistant grid with nodes $(x_i, y_j, z_k)$, $i = 1, \dots, N_x$, $j = 1, \dots, N_y$, and $k = 1, \dots, N_z$. This yields a finite element space $V_h$ with $N_h = 55815$ degrees of freedom. For the temporal discretization, the interval $[0,16]$ is partitioned into $64$ equidistant time steps. Spatial integrals are approximated using  Gauss-Lobatto quadrature. For time integration, we employ the $\zeta$-scheme with $\zeta = 0.5$, corresponding to the Crank-Nicolson method.
Following~\cite{klein2021sequential, seydel2017identifying}, we model the stored energy function as
\begin{equation*}
    C(x,Y) = \sum_{i=0}^{N_x} \sum_{j=0}^{N_y} \sum_{k=0}^{N_z} 
    q_{i j k}\, \nu_{i j k}(x)\, \hat{C}(Y),
\end{equation*}
with coefficients $q_{i j k} > 0$. Here, $\hat{C}(Y)$ denotes an isotropic stored energy density. The functions $\nu_{i j k}$ are linear tensor-product B-splines used for the finite element discretization. They form a partition of unity, i.e.,
\begin{equation*}
    \sum_{i,j,k} \nu_{i j k}(x) = 1, \qquad x \in \Omega.
\end{equation*}
To model the elastic material, we adopt the standard linear elastic model. The stored energy density is then given by
\begin{equation*}
    \hat{C}(Y) = \frac{\lambda}{2} [\operatorname{tr}(Y)]^2 + \mu \|Y_{\mathrm{sym}}\|_F^2,
\end{equation*}
where
\begin{equation*}
    Y_{\mathrm{sym}} \coloneqq \tfrac{1}{2}(Y + Y^\top),
\end{equation*}
and $\lambda, \mu > 0$ are the Lamé parameters. For the material parameters, we choose ${\lambda = \SI{2.18e10}{\newton\per\square\meter}}$ and $\mu = \SI{1.12e10}{\newton\per\square\meter}$. The density is set to $\rho = \SI{2.70e3}{\kilogram\per\cubic\meter}$, cf.~\cite{lechleiter2017identifying}. A defect in the material is represented by $q_{i j k} \neq 1$. Since only delamination at the upper surface is considered, we impose
\begin{equation*}
    q_{i j k} = 1, \qquad \text{for all } i \geq 1,\; j,\; k,
\end{equation*}
i.e., only coefficients associated with the upper layer may vary. Hence, the parameters of interest reduce to
\begin{equation*}
    q_p := q_{0 j k}, \qquad p = 1, \dots, N_Q, \qquad N_Q \coloneqq N_y N_z = 3721.
\end{equation*}
It can be verified directly that $C(x,Y)$ satisfies Assumption~\ref{ass:ass_to_stored_energy}. In particular, the inequalities~\eqref{eq:coercivity} hold for $\hat{C}(Y)$ with
\begin{equation*}
    \kappa_p = 2\mu, \qquad \mu_p = 3\lambda + 2\mu.
\end{equation*}
Thus, we can define the set of admissible parameters as
\begin{equation*}
    \Qad \coloneqq \left\{ q \in \Q \,\middle|\, 10^{-20} \leq q_p \leq 10^{20} \;\; \text{for all } p = 1, \dots, N_Q \right\}.
\end{equation*}

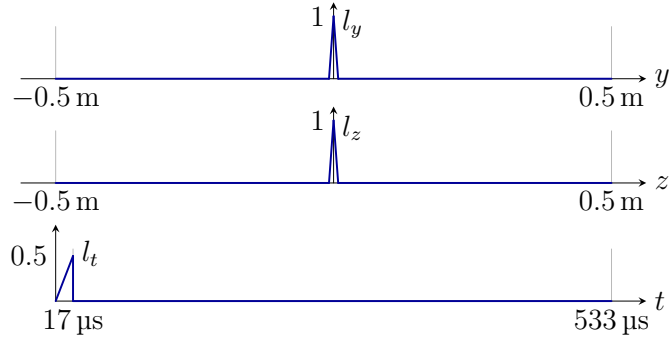
\begin{figure}
    \centering
    \resizebox{0.7\linewidth}{!}{%
        \begin{tikzpicture}[
    >=stealth,
    line cap=round,
    line join=round,
    font=\large,
    every node/.style={font=\large}
]

\colorlet{plotcolor}{blue!60!black}

\def\xL{-4.5}
\def\xR{4.5}
\def\xtick{4.0}
\def\xeps{0.066}

\def\ymax{1.1}
\def\ypeak{0.9}
\def\yhalf{0.65}
\def\h{0.75}
\def\vsep{1.5cm}

\begin{scope}[yshift=2*\vsep]
  \draw[->] (\xL,0) -- (\xR,0) node[right] {$y$};
  \draw[->] (0,0) -- (0,\ymax);

  \draw[gray!55] (-\xtick,0) -- (-\xtick,\h);
  \draw[gray!55] ( \xtick,0) -- ( \xtick,\h);

  \node[below] at (-\xtick,0) {$-\SI{0.5}{\meter}$};
  \node[below] at ( \xtick,0) {$\SI{0.5}{\meter}$};
  \node[left]  at (0,\ypeak) {$1$};
  \node[right] at (0,0.88*\ypeak) {$l_y$};

  \draw[plotcolor,thick] (-\xtick,0) -- (-\xeps,0);
  \draw[plotcolor,thick] (-\xeps,0) -- (0,\ypeak);
  \draw[plotcolor,thick] (0,\ypeak) -- (\xeps,0);
  \draw[plotcolor,thick] (\xeps,0) -- (\xtick,0);
\end{scope}

\begin{scope}[yshift=\vsep]
  \draw[->] (\xL,0) -- (\xR,0) node[right] {$z$};
  \draw[->] (0,0) -- (0,\ymax);

  \draw[gray!55] (-\xtick,0) -- (-\xtick,\h);
  \draw[gray!55] ( \xtick,0) -- ( \xtick,\h);

  \node[below] at (-\xtick,0) {$-\SI{0.5}{\meter}$};
  \node[below] at ( \xtick,0) {$\SI{0.5}{\meter}$};
  \node[left]  at (0,\ypeak) {$1$};
  \node[right] at (0,0.88*\ypeak) {$l_z$};

  \draw[plotcolor,thick] (-\xtick,0) -- (-\xeps,0);
  \draw[plotcolor,thick] (-\xeps,0) -- (0,\ypeak);
  \draw[plotcolor,thick] (0,\ypeak) -- (\xeps,0);
  \draw[plotcolor,thick] (\xeps,0) -- (\xtick,0);
\end{scope}

\begin{scope}[shift={(-4.0,-0.2)}]
  \draw[->] (0,0) -- (8.5,0) node[right] {$t$};
  \draw[->] (0,0) -- (0,\ymax);

  \draw[gray!55] (0.25,0) -- (0.25,\h);
  \draw[gray!55] (8.00,0) -- (8.00,\h);

  \node[left]  at (0,\yhalf) {$0.5$};
  \node[below] at (0.25,0) {$\SI{17}{\micro\second}$};
  \node[right] at (0.25,1.06*\yhalf) {$l_t$};
  \node[below] at (8.00,0) {$\SI{533}{\micro\second}$};

  \draw[plotcolor,thick] (0,0) -- (0.25,\yhalf);
  \draw[plotcolor,thick] (0.25,\yhalf) -- (0.25,0);
  \draw[plotcolor,thick] (0.25,0) -- (8.00,0);
\end{scope}

\end{tikzpicture}
    }
    \caption{\label{fig:excite_func}Spatial and temporal factor functions $l_y$, $l_z$, and $l_t$ of the excitation signal in the lateral ($y,z$) and temporal ($t$) domains, adapted from~\cite{seydel2017identifying}.}
\end{figure}
Our goal is to reconstruct the full parameter field $q_{0jk}$ by exciting the plate broadband signal of the form
\begin{equation*}
    l(t,x,y,z) = l_t(t)l_y(y)l_z(z)
    \begin{pmatrix}
    1 \\
    0 \\
    0
    \end{pmatrix}
\end{equation*}
emitted at the center of the plate the induced displacement field measured, cf.~\cite{seydel2017identifying}. The functions $l_t, l_y, l_z$ are shown in Figure~\ref{fig:excite_func}. We consider three different observation scenarios. As a first benchmark example we assume that the full displacement field is observed (\textit{Experiment 1}), i.e., the observation operator $\opall$ is given by the canonical embedding \eqref{eq:full_data_obs_op}. As a second class of experiments (\textit{Experiment 2}), the parameters are reconstructed from partial measurements of the displacement field on the upper surface ($x = -0.1$). These measurements are modeled by the observation operator \eqref{eq:sensor_data_obs_op}, where $g_c \in L^2(\partial\Omega;\mathbb{R}^3)$ denotes the finite-element basis functions associated with the corresponding boundary vertices. Hence, sensors are placed at
\begin{equation*}
  \{ (x,y,z) \mid x = -0.1,\ (y,z) \in E \}.
\end{equation*}
We consider two sensor configurations:

\textit{(a) Grid configuration}: Sensors are placed on a regular grid overing the measurement surfaces uniformly. The index set is
\begin{equation*}
  E_{\mathrm{grid}}
  := \{-14,-10,\dots,10,14\} \times \{-14,-10,\dots,10,14\},
\end{equation*}
resulting in $N_c = 64$ measurement points.

\textit{(b) Edge–only configuration}: Sensors are placed along the boundary lines of the surface.  
In this case $N_c = 4 \cdot 28$, with
\begin{equation*}
  E_{\mathrm{edge}}
  :=
    \bigl( \{-14,14\} \times \{-14,\dots,14\} \bigr)
    \,\cup\,
    \bigl( \{-14,\dots,14\} \times \{-14,14\} \bigr).
\end{equation*}

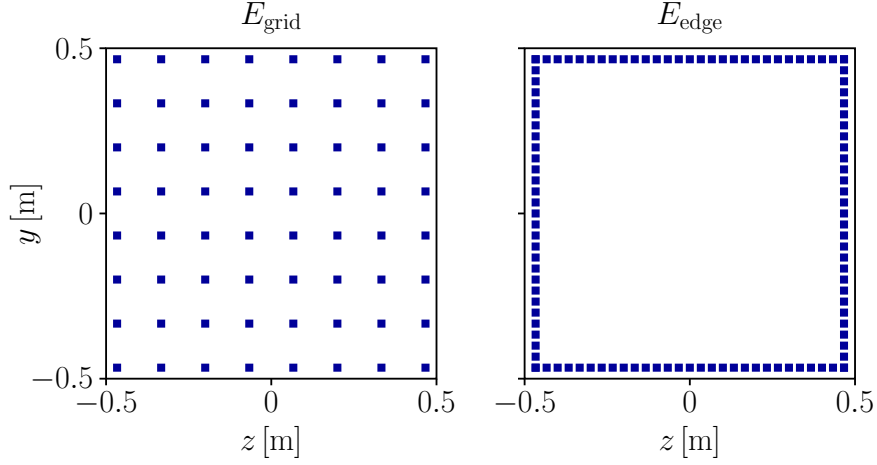
\begin{figure}
    \centering
    \resizebox{0.9\textwidth}{!}{
        \begin{tikzpicture}
    \begin{groupplot}[
        group style={
            group size=2 by 1,
            horizontal sep=2cm
        },
        width=7.5cm,
        height=7.5cm,
        scale only axis,
        xmin=-0.5, xmax=0.5,
        ymin=-0.5, ymax=0.5,
        axis equal image,
        axis lines=box,
        xtick={-0.5,0,0.5},
        ytick={-0.5,0,0.5},
        tick align=outside,
        tick pos=left,
        enlargelimits=false,
        tick style={line width=1.2pt, draw=black},
        label style={font=\huge},
        tick label style={font=\huge},
        title style={font=\huge}
    ]
    
    \nextgroupplot[
        title={$E_{\text{grid}}$},
        line width=1.2pt,
        xlabel={$z\,[\mathrm{m}]$},
        ylabel={$y\,[\mathrm{m}]$},
        ylabel style={yshift=-15pt},
    ]
    
    \foreach \yy in {-14,-10,...,14} {
      \foreach \zz in {-14,-10,...,14} {
        \addplot[
            only marks,
            mark=square*,
            mark size=2.0pt,
            color=blue!60!black
        ] coordinates {(\zz/30,\yy/30)};
      }
    }
    
    \nextgroupplot[
        title={$E_{\text{edge}}$},
        line width=1.2pt,
        xlabel={$z\,[\mathrm{m}]$},
        ylabel={},
        yticklabels=\empty
    ]
    
    \foreach \yy in {-14,14} {
      \foreach \zz in {-14,-13,...,14} {
        \addplot[
            only marks,
            mark=square*,
            mark size=2.0pt,
            color=blue!60!black
        ] coordinates {(\zz/30,\yy/30)};
      }
    }
    
    \foreach \yy in {-14,-13,...,14} {
      \foreach \zz in {-14,14} {
        \addplot[
            only marks,
            mark=square*,
            mark size=2.0pt,
            color=blue!60!black
        ] coordinates {(\zz/30,\yy/30)};
      }
    }
    
    \end{groupplot}
\end{tikzpicture}
    }
    \caption{Sensor configurations at $x = -0.1$ in the $y$–$z$ cross-section: uniform grid arrangement $E_\text{grid}$ (left) and edge-only arrangement $E_\text{edge}$ (right), showing the spatial distribution of measurement points.}
\end{figure}

\subsection{Experiments}
The synthetic noisy measurements are generated by first computing the exact solution ${u_h^\mathsf{e} \in V_h^K}$ corresponding to $q_h^\mathsf{e} \in Q_h$, and then adding uniformly distributed random noise ${\xi \in C_h^K \setminus \{0\}}$ with magnitude $\delta > 0$. Specifically, we define
\begin{equation*}
y^\delta \coloneqq \mathcal{C}_h u_h^\mathsf{e} + \delta\frac{\xi}{\|\xi\|_{C_h^K}}.
\end{equation*}
For each experiment, we compare the full-order IRGNM (FOM-IRGNM) and the trust-region IRGNM (TR-IRGNM), using $\epsPODstate = 10^{-3}$ as \gls{POD}-tolerance in the enrichment. In all setups, the reconstruction uses $q^{(0)} \equiv 1$ as an initial guess, corresponding to a defect-free plate. This also serves as the reference parameter for the Tikhonov regularization term, i.e., $q_\circ \equiv 1$. The regularization constant is initialized as $10^{-5}$. For all experiments, we fix $\theta = 0.4$ and  $\Theta = 1.95$, which control the adaptive update of the regularization parameter $\alpha > 0$. For the trust-region strategy, we choose
\begin{align*}
\eta^{(0)} &= 0.5, \quad \eta_\text{max} = 2, 
\quad \beta_2 = 0.75, \quad \beta_3 = 0.5.
\end{align*}
The discretized linearized problems~\eqref{prob:lin_prob_FOM} and \eqref{prob:lin_prob_rom} were solved by a projected gradient descent method. The step size was selected via a Barzilai--Borwein-type line search~\cite{azmi2023nonmonotone}. The optimization procedure was terminated when the first-order optimality condition was satisfied, when a maximum of $250$ iterations was reached, or when the relative change in the objective fell below $10^{-4}$. The parameter fields obtained from the FOM-IRGNM are denoted by $q^{\text{\footnotesize FOM}} \coloneqq q_h^{(\termidx)}$, and those obtained from the TR-IRGNM by $q^{\text{\footnotesize TR}} \coloneqq q_r^{(\termidx)}$.

The construction of the computational grid and the assembly of the full-order matrices are implemented in \texttt{C++} using the finite element library \texttt{deal.II} (version~9.6.0) \cite{2024:africa.arndt.ea:deal}. The time-stepping solver and the optimization and model order reduction algorithms are implemented in \texttt{Python} using the \texttt{pyMOR} framework~\cite{doi:10.1137/15M1026614}. The source code is available at \cite{klein2026adaptive}. All experiments were conducted on the PALMA~II high-performance computing cluster at the University of Münster, funded by the German Research Foundation (DFG, INST~211/667-1). Computations were performed on nodes equipped with Intel Skylake Gold~6140 CPUs clocked at 2.30~GHz and 92~GB RAM.

\subsection{Large Inclusion}
\label{ssec:large_inclusion}
In the first example, we consider a single large rectangular inclusion embedded into an otherwise homogeneous material. We set $q_{ijk} = \tilde{q}$ for all indices $i,j,k$ such that the support of $\nu_{ijk}$ is contained in $\{-0.1\} \times [-10,-5] \times [-10,10]$, see Figure~\ref{fig:q_exact_large_inclusions}. We compare two scenarios in which the material properties within this region are modified: in the first case (I), the material is locally stiffened with $\tilde{q} = 2.0$, whereas in the second case (II), it is softened with $\tilde{q} = 0.5$. In both cases, we assume a relative noise level of $1\%$, i.e.,
\begin{equation*}
\delta = 10^{-2} \,\|\mathcal{C}_h u_h^{\mathsf{e}}\|_{\mathcal{C}_h^K}.   
\end{equation*}
We choose $\tau = 1.1$ for Experiments~1 and~2(a), and $\tau = 2.0$ for Experiment~2(b). All runs satisfy the stopping criterion~\eqref{eq:discrepancy_principle}.

\begin{figure}
    \centering
    \includegraphics[width = 0.9\textwidth]{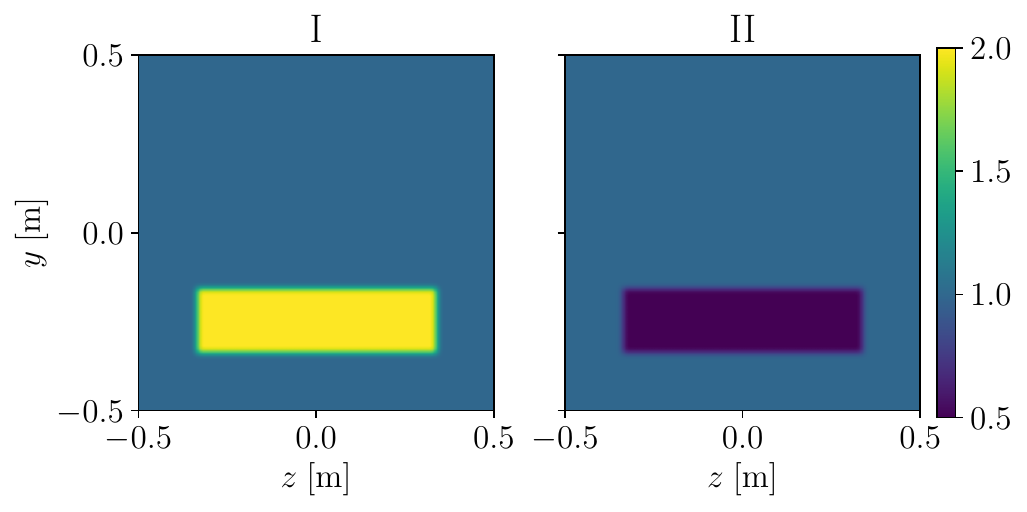}
    \caption{Exact parameter $\qex$ for a large square-shaped inclusion located in $\{-0.1\} \times [-10,-5] \times [-10,10]$. The parameter takes the constant value $\tilde{q}$; Left: stiff inclusion with $\tilde{q}=2.0$. Right: soft inclusion with $\tilde{q}=0.5$.}
    \label{fig:q_exact_large_inclusions}
\end{figure}

Figures~\ref{fig:elasticity_alu_large_patch_stiff} and~\ref{fig:elasticity_alu_large_patch_soft} show the reconstructed parameters for the stiff (I) and soft (II) inclusion, respectively, obtained with FOM-IRGNM (left) and TR-IRGNM (middle), together with the corresponding difference field (right) for the various data sets. 
Table~\ref{tab:results_large_inclusions} summarizes the corresponding performance metrics. In particular, the relative reconstruction error in the $Q_h$-norm,
\begin{equation*}
    \frac{\|q^{\text{\footnotesize FOM}} - q^{\text{\footnotesize TR}}\|_{Q_h}}
         {\|q^{\text{\footnotesize FOM}}\|_{Q_h}},
\end{equation*}
as well as the total runtime and the resulting speed-up relative to the full-order IRGNM. Furthermore, we report the number of full-order solves of the (linearized) PDE, the final dimensions of the reduced parameter and state spaces $(n_Q, n_V)$, the number of outer IRGNM iterations, and the total number of inner iterations accumulated over all outer iterations.

\begin{table}[b]
\setlength{\tabcolsep}{3.5pt}
\centering
\small
\begin{tabularx}{\textwidth}{
l l
c c
c
c
c c
c c
}
\toprule
Setup & Method
& $Q_h$ rel. err. & time [s]
& speed-up
& \#FOM sol.
& $n_Q$ & $n_V$
& o. iter. & tot. iter. \\
\midrule
\multicolumn{10}{l}{\textbf{Experiment 1:} $\mathcal{C} = \opall$}\\
\midrule
I & FOM & -- & 50114 & -- & 1393 & -- & -- & 24 & 24 \\
  & TR & 2.16e-02 & 6161 & 8.13 & 66 & 16 & 1004 & 16 & 45 \\
   
II & FOM & --        & 34282 & --    & 1018 & -- & --   & 18 & 18 \\
   & TR  & 2.18e-02 & 3020  & 11.35 & 35   & 9  & 842  & 8  & 34 \\

\midrule

\multicolumn{10}{l}{\textbf{Experiment 2 (a):} $\mathcal{C} = \opsensors, E = E_{\mathrm{grid}}$}\\
\midrule
I & FOM & -- & 59838 & -- & 1502 & -- & -- & 28 & 28 \\
  & TR & 4.43e-02 & 6369 & 9.39 & 78 & 19 & 800 & 19 & 61 \\
   
II & FOM & -- & 45358 & -- & 1181 & -- & -- & 21 & 21 \\
   & TR & 2.57e-02 & 3173 & 14.29 & 53 & 12 & 646 & 12 & 44 \\

\midrule

\multicolumn{10}{l}{\textbf{Experiment 2 (b):} $\mathcal{C} = \opsensors, E = E_{\mathrm{edge}}$}\\
\midrule
I  & FOM & -- & 44101 & -- & 1310 & -- & -- & 27 & 27 \\
   & TR & 2.17e-02 & 6076 & 7.26 & 68 & 15 & 924 & 15 & 58 \\
II & FOM & -- & 33677 & -- & 1059 & -- & -- & 22 & 22 \\
   & TR & 1.76e-02 & 2569 & 13.11 & 45 & 10 & 569 & 10 & 46 \\
\bottomrule
\end{tabularx}
\caption{Comparison of \gls{FOM}-\gls{IRGNM} and \gls{TR}-\gls{IRGNM} for large inclusions. Reported are relative errors in $Q_h$, runtimes, speed-ups, number of FOM solves, reduced dimensions $(n_Q, n_V)$, and iteration counts.}
\label{tab:results_large_inclusions}
\end{table}

The influence of the available data on the quality of the reconstructed parameter field is clearly visible. When the full displacement field is observed (Experiment~1), both the position and shape of the large inclusion, as well as the local values of the Lamé coefficients, are reconstructed with high accuracy. In contrast, reconstructions based on sensor measurements exhibit noticeable artifacts. While the general location of the inclusion can still be identified, its shape and the values of the reconstructed parameter field deviate significantly. In particular, for the sensor grid, the artifacts appear to be random, which suggests that they are primarily induced by noise. This indicates that the choice $\tau = 1.1$ in the stopping criterion may already be too small. The relative errors reported in Table~\ref{tab:results_large_inclusions} (Column~3) show that, across all experiments, the parameter fields reconstructed by the trust-region approach closely match those obtained with the full-order IRGNM. The relative reconstruction error in the $Q_h$-norm remains below $5\%$ for all runs. It can thus be concluded that these artifacts are not caused by the \gls{RB} surrogate. This suggests that the trust-region strategy provides sufficiently accurate surrogate control for reliable parameter reconstruction.

Regarding the computational time, TR-IRGNM is generally faster than FOM-IRGNM primarily due to the reduced number of full-order PDE solves. In fact, the majority of full-order solves can be attributed to the computation of the primal and adjoint snapshots $u_h(q_r^{\iteridx})$ and $p_h(q_r^{\iteridx})$. The observed speed-ups range between $7.16$ and $14.29$, with larger gains observed for the soft inclusion (II) than for the hard one (I). Compared with the pure online efficiency of standalone RB-ROMs, these speed-ups are relatively modest, as such methods often achieve speed-ups of the order $\mathcal{O}(10^1)$ to $\mathcal{O}(10^2)$.However, such a comparison is not entirely appropriate, since \gls{TR}-\gls{IRGNM} does not involve a distinct offline phase. Instead, the trust-region optimization alternates between offline phases (re-adaptation of the \gls{RB}-\gls{ROM}) and online phases (execution of the Gauss--Newton iterations). When compared to other trust-region-type optimization methods, the observed speed-ups appear to be reasonable~\cite{kartmann_adaptive_2023, kartmann2025adaptivereducedbasistrust, keil2024relaxed,keil2021non}.

Moreover, we observe here, and discuss further in Section~\ref{ssec:dep_on_noise_level}, that the achieved speed-ups are negatively correlated with the total number of iterations. This behavior is to be expected: a larger number of iterations required for convergence leads to more frequent updates of the \gls{RB}-\gls{ROM} and, consequently, to higher reduced-basis dimensions. This, in turn, both increases the cost of the offline phases and slows down the online evaluation.

\subsection{Point defect}

\begin{figure}
    \centering
    \includegraphics[width = \textwidth]{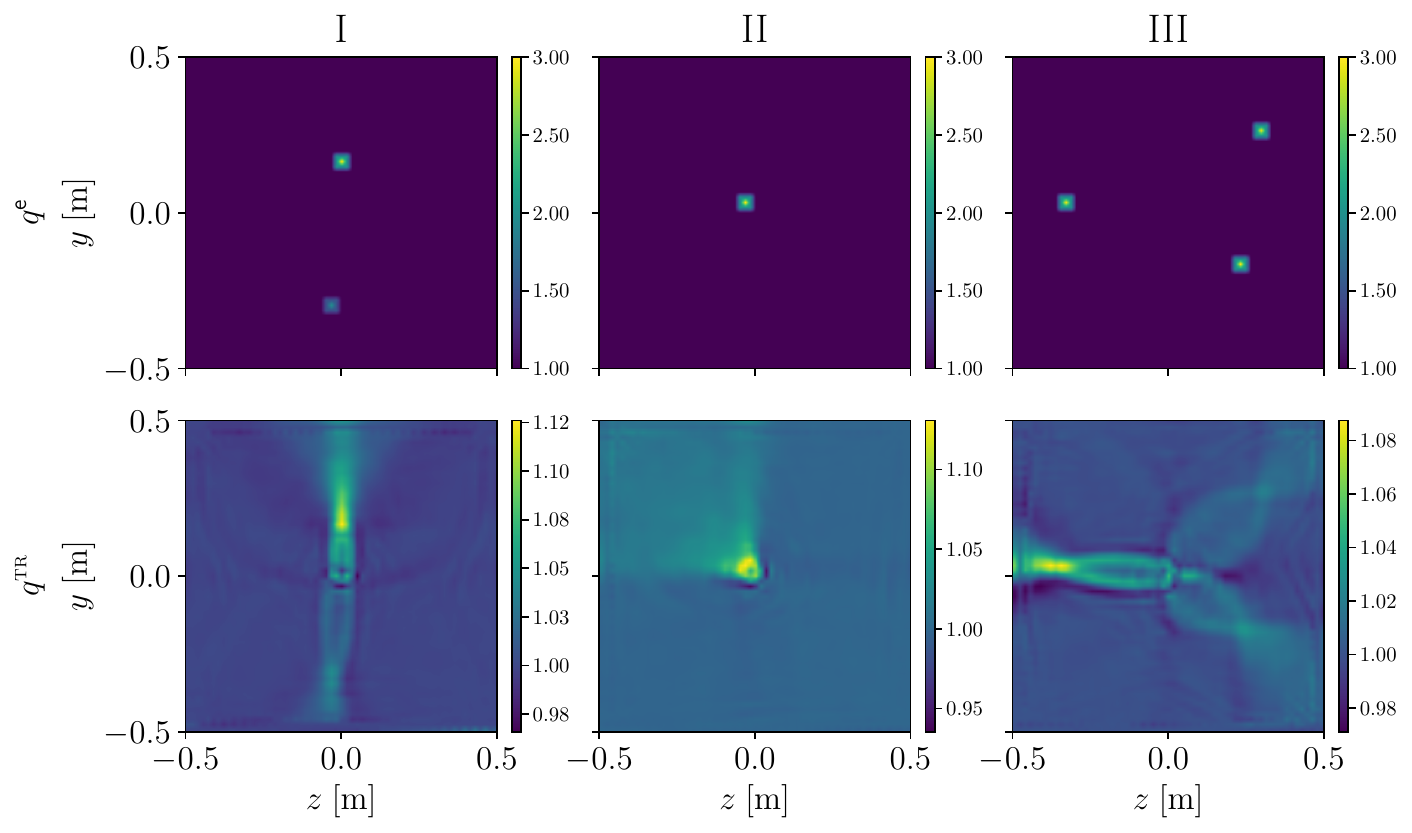}
    \caption{Exact parameter field $\qex$ (top) and reconstructed parameter fields $q^{\text{\footnotesize TR}}$ (bottom) by \gls{TR}-\gls{IRGNM} for the point defect problems.}
    \label{fig:point_defects_q_exact_and_q_TR}
\end{figure}

As a second example, we investigate the performance of the \gls{TR}-\gls{IRGNM} in scenarios inspired by \cite{seydel2017identifying, klein2021sequential, binder2015defect}, namely configurations with highly localized defects. To this end, we consider three configurations of defects. For the early stopping we have considered again $\tau = 1.1$ for Experiments~1 and~2(a), and $\tau = 2.0$ for Experiment~2(b).  

For the first setup~I, we introduce two localized inclusions. The first inclusion is assigned the value $3.0$ and is centered at $(5.0, 0.0)$, while the second has value $2.0$ at $(-9.0, -1.0)$. The inclusions are constructed according to the following rule: a single central node is assigned the full inclusion value, while its eight directly adjacent neighboring nodes are assigned values halfway between the inclusion value and the background value. This yields a small transition region around each inclusion, thereby yielding a localized but smooth transition to the background. The remaining configurations are constructed analogously. In setup~II, we consider a single defect centered at $(1.0, -1.0)$ with value $3.0$. In setup~III, we consider three defects located at $(1.0, -10.0)$, $(-5.0, 7.0)$, and $(8.0, 9.0)$, all assigned the value $3.0$.

\begin{table}
\setlength{\tabcolsep}{3.5pt}
\centering
\small
\begin{tabularx}{\textwidth}{
l l
c c
c
c
c c
c c
}
\toprule
Setup & Method
& $Q_h$ rel. err. & time [s]
& speed-up
& \#FOM sol.
& $n_Q$ & $n_V$
& o. iter. & tot. iter. \\
\midrule
\multicolumn{10}{l}{\textbf{Experiment 1:} $\mathcal{C} = \opall$}\\
\midrule
I   & FOM & -- & 12039 & -- & 329 & -- & -- & 6 & 6 \\
    & TR & 3.70e-03 & 854 & 14.09 & 18 & 4 & 282 & 4 & 10 \\
II  & FOM & -- & 11787 & -- & 325 & -- & -- & 6 & 6 \\
    & TR & 4.15e-03 & 1145 & 10.29 & 22 & 5 & 332 & 5 & 14 \\
III & FOM & -- & 7574 & -- & 225 & -- & -- & 4 & 4 \\
    & TR & 2.16e-03 & 711 & 10.65 & 14 & 3 & 229 & 3 & 6 \\
\midrule

\multicolumn{10}{l}{\textbf{Experiment 2 (a):} $\mathcal{C} = \opsensors, E = E_{\mathrm{grid}}$}\\
\midrule
I   & FOM & -- & 11046 & -- & 374 & -- & -- & 6 & 6 \\
    & TR & 3.41e-03 & 1074 & 10.28 & 21 & 5 & 272 & 5 & 11 \\
II  & FOM & -- & 18079 & -- & 495 & -- & -- & 11 & 11 \\
    & TR & 2.53e-03 & 1139 & 15.87 & 22 & 5 & 272 & 5 & 15 \\
III & FOM & -- & 14215 & -- & 378 & -- & -- & 6 & 6 \\
    & TR & 4.64e-03 & 754 & 18.84 & 17 & 4 & 236 & 4 & 10 \\
\midrule

\multicolumn{10}{l}{\textbf{Experiment 2 (b):} $\mathcal{C} = \opsensors, E = E_{\mathrm{edge}}$}\\
\midrule
I   & FOM & -- & 14238 & -- & 466 & -- & -- & 11 & 11 \\
    & TR & 2.75e-03 & 672 & 21.17 & 14 & 3 & 198 & 3 & 14 \\
II  & FOM & -- & 13425 & -- & 472 & -- & -- & 12 & 12 \\
    & TR & 3.38e-03 & 612 & 21.91 & 14 & 3 & 198 & 3 & 15 \\
III & FOM & -- & 12674 & -- & 420 & -- & -- & 9 & 9 \\
    & TR & 2.02e-03 & 712 & 17.80 & 14 & 3 & 198 & 3 & 12 \\
\bottomrule
\end{tabularx}
\caption{Comparison of \gls{FOM}-\gls{IRGNM} and \gls{TR}-\gls{IRGNM} for point defect problems. Metrics are as in Table~\ref{tab:results_large_inclusions}.}
\label{tab:results_point_defects}

\end{table}

All runs converged with respect to the prescribed stopping criteria. Figure~\ref{fig:point_defects_q_exact_and_q_TR} shows the exact parameter fields $\qex$ for the three setups (top) and the corresponding reconstructed parameter fields (bottom) obtained using \gls{TR}-\gls{IRGNM} using boundary sensor measurements (Experiment~2(b)). Additionally, the reconstructions for setup~I based on full-field and grid sensor data are shown in Figure~\ref{fig:elasticity_alu_point_defects_I}. The corresponding plots for the other setups are omitted, since the qualitative behavior is consistent across all three cases. Table~\ref{tab:results_point_defects} reports the same metrics as in the previous section.

The observations for the point-defect configurations are largely consistent with those obtained for the large inclusions. However, \gls{TR}-\gls{IRGNM} performs slightly better in this regime: the relative errors with respect to \gls{FOM}-\gls{IRGNM} are smaller than in Section~\ref{ssec:large_inclusion} (below $1\%$), while for Experiments~2(a) and~2(b) higher speedups are achieved, reaching up to $21.91$. This improvement can be attributed to the smaller number of iterations required for convergence, which reduces the relative cost of the offline phase and results in smaller final reduced-basis dimensions.

\subsection{Dependence on noise level}
\label{ssec:dep_on_noise_level}

As a final experiment, we investigate the performance of the \gls{TR}-\gls{IRGNM} for different noise levels. As a benchmark, we consider setup~I for the detection of point defects introduced in the previous section. The noise level is set as
\begin{equation*}
\delta = \relnoise \, \|\mathcal{C}_h u_h^{\mathsf{e}}\|_{\mathcal{C}_h^K},
\end{equation*}
with $\relnoise \in \{0,\; 10^{-3},\; 5 \cdot 10^{-3},\; 10^{-2}\}$ and $\tau = 2.0$. Larger noise levels were also tested, however, in these cases the discrepancy principle was already satisfied at the initial iterate $q_r^{(0)}$.

\begin{figure}[b]
    \centering
    \includegraphics[width=\textwidth]{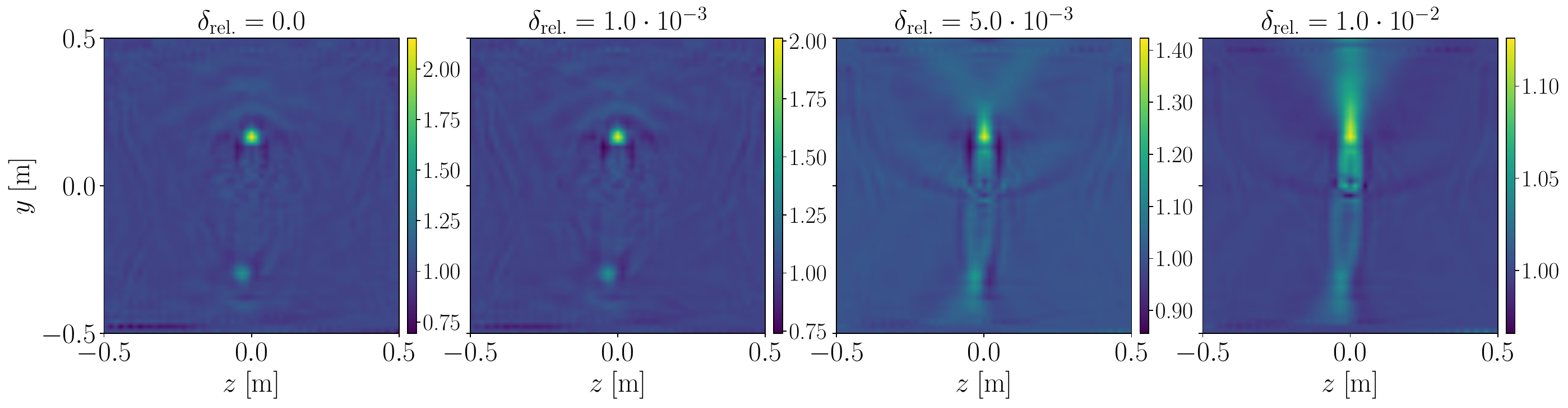}
    \caption{Reconstructed parameter fields for different relative noise levels obtained by \gls{TR}-\gls{IRGNM}.}
    \label{fig:reconst_param_noise_level}
\end{figure}

We employ edge sensor measurements (Experiment~2(b)) and compare the performance of the \gls{FOM}-\gls{IRGNM} and the \gls{TR}-\gls{IRGNM}. All runs were terminated after 24 hours if convergence had not been achieved earlier. For $\relnoise \in \{10^{-3}, 5\cdot10^{-3},\,10^{-2}\}$, both methods converged. For $\relnoise = 0$, neither method converged within the allotted time. Consequently, the speed-up values reported in Table~\ref{tab:results_noise_level} are meaningful only for $\relnoise \in \{5\cdot10^{-3},\,10^{-2}\}$. For $\relnoise = 10^{-3}$, the reported value represents an upper bound, and for $\relnoise = 0$ no meaningful speed-up can be assigned.

\begin{table}
\setlength{\tabcolsep}{3.5pt}
\centering
\small
\begin{tabularx}{\textwidth}{
l l
c c
c
c
c c
c c
}
\toprule
$\relnoise$ & Method
& $Q_h$ rel.\ err. & time [s]
& speed-up
& \#FOM sol.
& $n_Q$ & $n_V$
& o.\ iter. & tot.\ iter. \\
\midrule
0.0  & FOM & -- & 80577 & -- & 2622 & -- & -- & 45 & 45 \\
     & TR & 1.49e-02 & 86122 & -- & 229 & 57 & 2216 & 56 & 303 \\
1e-3 & FOM & -- & 57591 & -- & 1798 & -- & -- & 34 & 34 \\
     & TR & 5.21e-03 & 18907 & 3.05 & 130 & 32 & 1213 & 32 & 161 \\
5e-3 & FOM & -- & 23121 & -- & 789 & -- & -- & 18 & 18 \\
     & TR & 7.70e-03 & 1185 & 19.50 & 22 & 5 & 268 & 5 & 30 \\
1e-2 & FOM & -- & 14195 & -- & 466 & -- & -- & 11 & 11 \\
     & TR & 2.75e-03 & 605 & 23.45 & 14 & 3 & 198 & 3 & 14 \\
\bottomrule
\end{tabularx}
\caption{Comparison of \gls{FOM}-\gls{IRGNM} and \gls{TR}-\gls{IRGNM} for a point defect, setup I, with varying relative noise levels $\relnoise \in \{0, 10^{-3}, 5 \cdot 10^{-3}, 10^{-2}\}$.}
\label{tab:results_noise_level}
\end{table}

Figure~\ref{fig:reconst_param_noise_level} shows the reconstructed parameter fields obtained with the \gls{TR}-\gls{IRGNM}. As expected, the reconstruction quality improves as the noise level decreases: the two point defects become more sharply localized and the local Lam\'e parameters converge toward their true values. However, Table~\ref{tab:results_noise_level} shows that the computational efficiency of the \gls{TR}-\gls{IRGNM} deteriorates as the noise level decreases.

\begin{figure}[b]
    \centering
    \includegraphics[width = \textwidth]{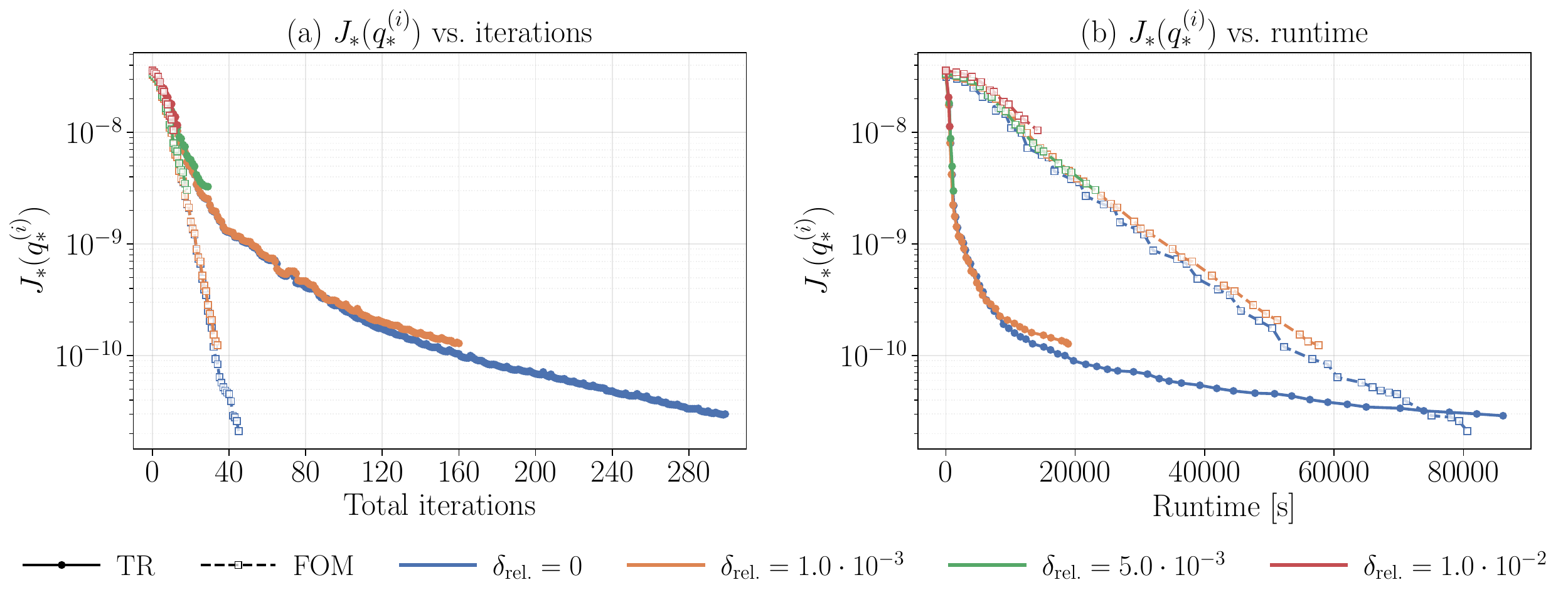}
    \caption{Objective $J_\ast(q_\ast^{(i)})$ plotted against total iterations (a) and against runtime (b). The colors indicate noise levels, while line styles and markers distinguish \gls{TR} and \gls{FOM}.}
    \label{fig:objective_comparison}
\end{figure}

To further analyse this behaviour, Fig.~\ref{fig:objective_comparison} shows the objective functional $J_\ast(q_\ast^{(i)})$ plotted against both the total number of iterations (a) and the total runtime (b). During the early iterations, the \gls{TR}-\gls{IRGNM} achieves a reduction in the objective comparable to that of the \gls{FOM}-\gls{IRGNM}. However, as the optimization progresses, the two methods diverge: the \gls{TR}-\gls{IRGNM} exhibits progressively smaller decreases in the objective, while the \gls{FOM}-\gls{IRGNM} maintains a more consistent convergence rate.

In addition, the runtime per outer iteration of \gls{TR}-\gls{IRGNM} increases substantially over the course of the optimization. Although the initial iterations are significantly faster compared to \gls{FOM}-\gls{IRGNM}, this advantage gradually diminishes over the iterations. Combined with progressively smaller reductions in the objective functional per iteration, this results in a gradual loss of efficiency of \gls{TR} - \gls{IRGNM}. In particular, for low noise levels where many iterations are required to reach convergence, \gls{TR}-\gls{IRGNM} eventually becomes less efficient than \gls{FOM}-\gls{IRGNM}, in terms of convergence rate and computational cost.

First, the reduced parameter space contains only a limited number of admissible decay directions compared to the full-order space. Consequently, the objective functional exhibits a less pronounced decrease per iteration in the \gls{TR}-\gls{IRGNM} than in the corresponding \gls{FOM}-\gls{IRGNM}. In other words, the reduced model restricts the set of feasible update directions, thereby limiting the achievable reduction of the objective functional in each iteration. The second effect is related to the state reduction for hyperbolic problems. It is well known that standard Galerkin-based reduced-basis methods face intrinsic difficulties for wave-like or transport-dominated dynamics, since the associated solution manifolds typically exhibit slow decay of the Kolmogorov $n$-width~\cite{greif2019decay, arbes2023kolmogorov}. Consequently, a comparatively large number of basis functions is required to obtain an accurate reduced-order approximation over larger regions of the parameter space, and the relevant dynamics cannot be represented efficiently in a low-dimensional reduced space.

\begin{figure}
\centering
\includegraphics[width = \textwidth]{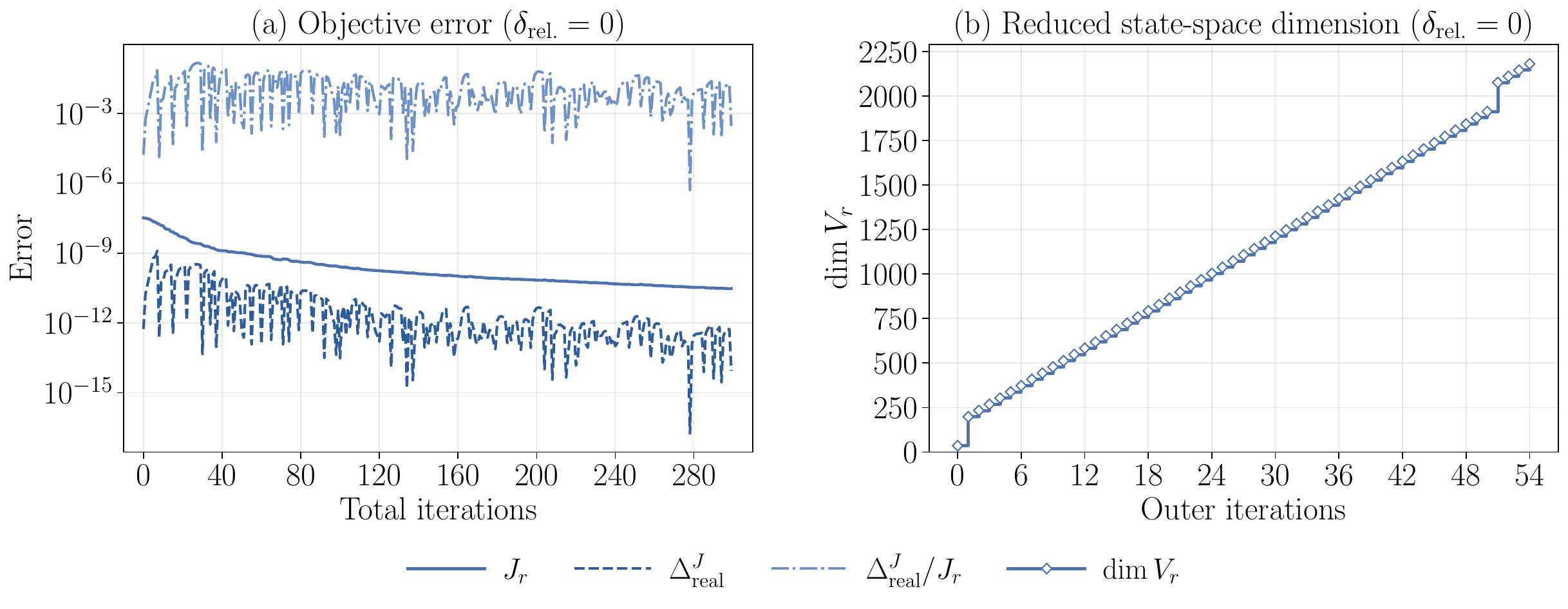}
\caption{(a) For the noise-free case $\delta_{\mathrm{rel.}}=0$, the reduced objective $J_r$ together with the corresponding absolute and relative errors.
(b) For the noise-free case $\delta_{\mathrm{rel.}}=0$, the evolution of the reduced state-space dimension $\dim V_r$ over the outer iterations.}
\label{fig:error_noise_0}
\end{figure}

To support this heuristic interpretation, Figure~\ref{fig:error_noise_0}(a) shows the absolute error
\begin{equation*}
\lvert J_h(q_r^{(i,l)}) - J_r(q_r^{(i,l)}) \rvert =: \Delta^J_{\mathrm{real}},
\end{equation*}
together with the corresponding relative error $\Delta^J_{\mathrm{real}} / J_r(q_r^{(i,l)})$ for the run with $\relnoise = 0$. Figure~\ref{fig:error_noise_0}(b) shows the evolution of the reduced state-space dimension $\dim V_r^\iteridx$ over the outer iterations. Figure~\ref{fig:error_noise_0}(a) shows that the relative objective error decreases sharply immediately after enrichment with the local snapshots $u_h(q_h^{(i)})$ and $p_h(q_h^{(i)})$, but subsequently increases again to a similar level (around $10^{-1}$) as the iterates move away from the corresponding snapshot parameter. Importantly, no systematic decrease of the relative error can be observed within the considered range of reduced dimensions. This indicates that the local enrichments do not yield a globally accurate surrogate model, since the reduced state space remains insufficiently rich to capture the dominant dynamics over the explored parameter region. As a consequence, additional basis enrichment provides only local improvements in approximation quality, as can be observed in Figure~\ref{fig:error_noise_0}(b). In each enrichment step, approximately $30-40$ new vectors are added to the reduced basis. No saturation can be observed, indicating that a substantial fraction of the newly generated snapshots is still required at every enrichment step in order to accurately represent the evolving local dynamics.

\section{Conclusion \& Outlook}
In this work, we have extended the numerical methodology introduced in \cite{kartmann_adaptive_2023, kartmann2025adaptivereducedbasistrust}, which employs adaptive \acrlong{RB} \acrlong{ROM}s within a trust-region-based \gls{IRGNM} procedure, to inverse problems governed by Cauchy’s equation in (hyper)elastic materials. To the best of our knowledge, this is the first application of this framework to hyperbolic PDE-constrained problems.

Numerical experiments motivated by \acrlong{SHM} demonstrate the feasibility of the proposed approach across a range of defect scenarios. Significant reductions in computational cost are achieved while maintaining reconstruction accuracy comparable to those of the full-order IRGNM, both for full-state observations and for sensor-based measurements. The trust-region mechanism provides reliable error control, ensuring close agreement between reduced and full-order solutions and accurate recovery of the main features of the parameter field. 

These findings underline the potential of the approach for more complex applications and real-time monitoring pipelines. At the same time, they reveal intrinsic limitations of classical reduced-basis methods, which restrict achievable speed-ups in more demanding regimes. Future work will therefore focus on the development of more expressive surrogate models, for instance based on data-driven or machine learning techniques, as well as on establishing rigorous convergence guarantees for the proposed framework.

\section*{Statements and declarations}
\bmhead{Funding} Benedikt Klein and Mario Ohlberger acknowledge funding from the Deutsche Forschungsgemeinschaft (DFG, German Research Foundation) under Germany's Excellence Strategy – EXC 2044/2 – 390685587, Mathematics Münster: Dynamics–Geometry–Structure. 

\bmhead{Competing interests} The authors have no relevant financial or non-financial interests to disclose.

\bibliography{biblio}

\newpage
\appendix

\renewcommand{\thefigure}{\thesection.\arabic{figure}}
\setcounter{figure}{0}

\section{Proofs}
\subsection{Proof of Theorem~\ref{thm:state_apost_error_est}}
\label{ssec:proof_state_error_est}

\begin{proof}
The proof follows a standard discrete energy argument; see, e.g.,~\cite{glas2020reduced,bernardi2005time}. For brevity, we introduce the notation
\begin{equation*}
e^{k-1+\zeta} := (1-\zeta)e^{k-1} + \zeta e^k,
\qquad
\dot e^{k-1+\zeta} := (1-\zeta)\dot e^{k-1} + \zeta \dot e^k.
\end{equation*}
Since the full-order solution satisfies the primal scheme exactly and the reduced solution satisfies the reduced scheme, the error satisfies, for all $k \in \mathbb{K}$,
\begin{equation}
\label{eq:error_eqs_fixed}
\begin{aligned}
\frac{1}{\Delta t}\bm M_H(\dot{\bm e}^k - \dot{\bm e}^{k-1})
+ \bm A_h(q_r)\bm e^{k-1+\zeta}
&= \bm r_{\mathrm{pr}}^k(u_r,\dot u_r;q_r), \\[1mm]
\frac{1}{\Delta t}(\bm e^k - \bm e^{k-1})
&= \dot{\bm e}^{k-1+\zeta}.
\end{aligned}
\end{equation}
We test the first equation in~\eqref{eq:error_eqs_fixed} with $\dot{\bm e}^{k-1+\zeta}$ and use the second equation to rewrite the stiffness term:
\begin{equation*}
(\dot{\bm e}^{k-1+\zeta})^\top \bm A_h(q_r)\bm e^{k-1+\zeta}
=
(\bm e^{k-1+\zeta})^\top \bm A_h(q_r)\dot{\bm e}^{k-1+\zeta}
=
\frac{1}{\Delta t}(\bm e^{k-1+\zeta})^\top \bm A_h(q_r)(\bm e^k - \bm e^{k-1}).
\end{equation*}
This yields the energy balance
\begin{equation}
\label{eq:energy_balance_fixed}
\frac{1}{\Delta t}(\dot{\bm e}^{k-1+\zeta})^\top \bm M_H(\dot{\bm e}^k - \dot{\bm e}^{k-1})
+
\frac{1}{\Delta t}(\bm e^{k-1+\zeta})^\top \bm A_h(q_r)(\bm e^k - \bm e^{k-1})
=
(\dot{\bm e}^{k-1+\zeta})^\top \bm r_{\mathrm{pr}}^k(u_r,\dot u_r;q_r).
\end{equation}
Next, we use the identity
\begin{equation*}
\big((1-\zeta)x + \zeta y\big)^\top B(y - x)
=
\frac{1}{2}\big(\|y\|_B^2 - \|x\|_B^2\big)
+
\left(\zeta - \frac{1}{2}\right)\|y - x\|_B^2,
\end{equation*}
which holds for any symmetric positive semidefinite matrix $B$. Since $\zeta \geq \tfrac{1}{2}$, this implies
\begin{equation}
\label{eq:weighted_identity}
\big((1-\zeta)x + \zeta y\big)^\top B(y - x)
\ge
\frac{1}{2}\big(\|y\|_B^2 - \|x\|_B^2\big).
\end{equation}
Applying~\eqref{eq:weighted_identity} with
\begin{equation*}
B = \bm M_H,\quad x = \dot{\bm e}^{k-1},\quad y = \dot{\bm e}^k,
\end{equation*}
and with
\begin{equation*}
B = \bm A_h(q_r),\quad x = \bm e^{k-1},\quad y = \bm e^k,
\end{equation*}
we obtain from~\eqref{eq:energy_balance_fixed}
\begin{equation*}
\frac{1}{2\Delta t}
\left(
\|\dot{\bm e}^k\|_{\bm M_H}^2 - \|\dot{\bm e}^{k-1}\|_{\bm M_H}^2
+
\|\bm e^k\|_{\bm A_h(q_r)}^2 - \|\bm e^{k-1}\|_{\bm A_h(q_r)}^2
\right)
\le
(\dot{\bm e}^{k-1+\zeta})^\top \bm r_{\mathrm{pr}}^k(u_r,\dot u_r;q_r).
\end{equation*}
Introducing the energy norm
\begin{equation*}
\|e^k\|_E^2
:=
\|\dot{\bm e}^k\|_{\bm M_H}^2
+
\|\bm e^k\|_{\bm A_h(q_r)}^2,
\end{equation*}
this yields
\begin{equation}
\label{eq:Ek_step}
\frac{1}{2\Delta t}\big(\|e^k\|_E^2 - \|e^{k-1}\|_E^2\big)
\le
(\dot{\bm e}^{k-1+\zeta})^\top \bm r_{\mathrm{pr}}^k(u_r,\dot u_r;q_r).
\end{equation}
By duality, we estimate the right-hand side as
\begin{equation*}
(\dot{\bm e}^{k-1+\zeta})^\top \bm r_{\mathrm{pr}}^k(u_r,\dot u_r;q_r)
\le
\|r_{\mathrm{pr}}^k(u_r,\dot u_r;q_r)\|_{H_h'}
\, \|\dot e^{k-1+\zeta}\|_{H_h}.
\end{equation*}
Using
\begin{equation*}
\dot e^{k-1+\zeta} = (1-\zeta)\dot e^{k-1} + \zeta \dot e^k,
\end{equation*}
and $\zeta \in [1/2,1]$, we obtain
\begin{equation*}
\|\dot e^{k-1+\zeta}\|_{H_h}
\le
(1-\zeta)\|\dot e^{k-1}\|_{H_h} + \zeta \|\dot e^k\|_{H_h}
\le
\|e^{k-1}\|_E + \|e^k\|_E.
\end{equation*}
Inserting this into~\eqref{eq:Ek_step} yields
\begin{equation*}
\|e^k\|_E^2 - \|e^{k-1}\|_E^2
\le
2\Delta t\, \|r_{\mathrm{pr}}^k(u_r,\dot u_r;q_r)\|_{H_h'}
\big(\|e^{k-1}\|_E + \|e^k\|_E\big).
\end{equation*}
Using the identity
\begin{equation*}
a^2 - b^2 = (a - b)(a + b),
\end{equation*}
with $a = \|e^k\|_E$ and $b = \|e^{k-1}\|_E$, we obtain
\begin{equation*}
(\|e^k\|_E - \|e^{k-1}\|_E)(\|e^k\|_E + \|e^{k-1}\|_E)
\le
2\Delta t\, \|r_{\mathrm{pr}}^k(u_r,\dot u_r;q_r)\|_{H_h'}
(\|e^k\|_E + \|e^{k-1}\|_E).
\end{equation*}
If $\|e^k\|_E + \|e^{k-1}\|_E = 0$, the estimate is trivial. Otherwise, dividing by this term yields
\begin{equation*}
\|e^k\|_E - \|e^{k-1}\|_E
\le
2\Delta t\, \|r_{\mathrm{pr}}^k(u_r,\dot u_r;q_r)\|_{H_h'}.
\end{equation*}
Since $e^0 = 0$ and $\dot e^0 = 0$, we have $\|e^0\|_E = 0$, and summation over $k' = 1, \dots, k$ gives
\begin{equation}
\label{eq:pointwise_bound_fixed}
\|e^k\|_E
\le
2\Delta t \sum_{k'=1}^k \|r_{\mathrm{pr}}^{k'}(u_r,\dot u_r;q_r)\|_{H_h'}.
\end{equation}
Finally, squaring~\eqref{eq:pointwise_bound_fixed}, summing over $k = 1, \dots, K$, and taking the square root yields
\begin{equation*}
\left(\sum_{k=1}^K \|e^k\|_E^2\right)^{1/2}
\le
2\left(
\sum_{k=1}^K
\left(
\Delta t \sum_{k'=1}^k \|r_{\mathrm{pr}}^{k'}(u_r,\dot u_r;q_r)\|_{H_h'}
\right)^2
\right)^{1/2}.
\end{equation*}
This proves~\eqref{eq:state_apost_error_est}.
\end{proof}

\subsection{Proof of Theorem~\ref{thm:J_apost_error_est}}
\label{ssec:proof_obj_error_est}
\begin{proof}
We write
\begin{equation*}
\label{eq:error_est_J_intermed_1}
\begin{aligned}
J_h(q_r) - J_r(q_r)
&=
\Delta t \sum_{k = 1}^K
\left[
\frac{1}{2} (\errprim{k}{\bm})^\top \bm C_h \errprim{k}{\bm}
+
(\errprim{k}{\bm})^\top \bm C_h \bigl(\bm u_r^k - \bm y_h^{\delta,k}\bigr)
\right] \\
&\le
\Delta t \sum_{k = 1}^K
\left[
\frac{\|\mathcal C\|^2_{\mathcal L(V,C)}}{2}\,\|\errprim{k}{}\|_{V_h}^2
+
\|\mathcal C u_r^k - y_h^{\delta,k}\|_{C_h}\,\|\errprim{k}{}\|_{V_h}
\right].
\end{aligned}
\end{equation*}
By the G{\aa}rding inequality~\eqref{eq:garding}, the Poincar\'e inequality, and Assumption~\ref{ass:ass_to_stored_energy}, there holds
\begin{equation*}
\|v\|_V^2 \le \frac{1}{a_{q_r}} \langle A(q_r)v,v\rangle_{V',V},
\qquad a_{q_r} = \tilde{\kappa}/C_P^2 > 0.
\end{equation*}
Hence,
\begin{equation*}
\|\errprim{k}{}\|_{V_h}^2
\le
\frac{1}{a_{q_r}} (\errprim{k}{\bm})^\top \bm A_h(q_r)\errprim{k}{\bm}.
\end{equation*}
Applying Cauchy--Schwarz and using the definition of $J_r(q_r)$ yields~\eqref{eq:J_apost_error_est}.
\end{proof}


\newpage
\section{Reconstructed parameter fields}
\subsection{Large Inclusions}

\begin{figure}[H]
    \centering
    \includegraphics[width=\textwidth]{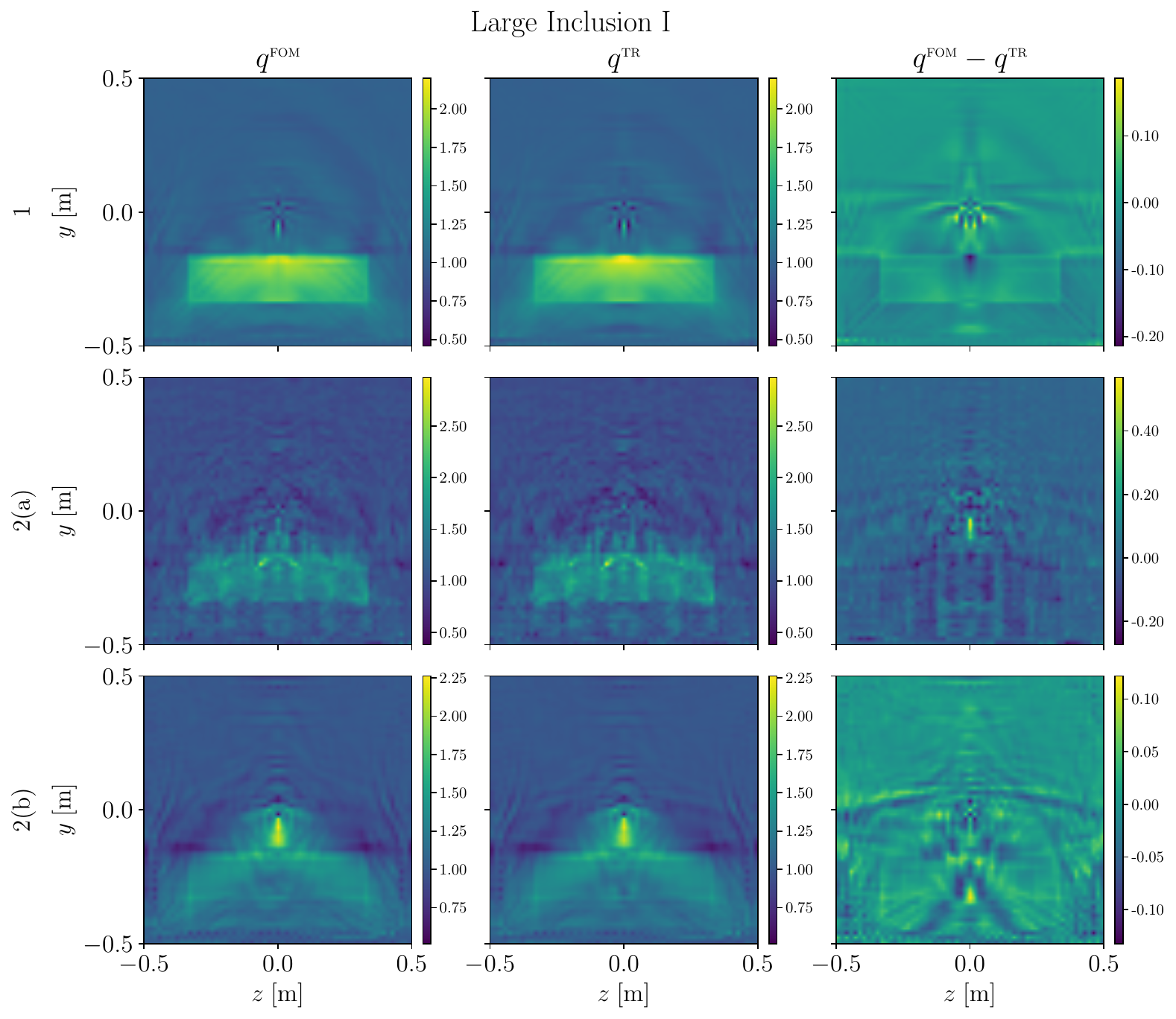}
    \caption{
    Reconstructed parameter field for the large inclusion I (stiffened case, $\tilde{q}=2.0$).
    Rows correspond to Experiments~1,~2(a), and~2(b), while columns show the parameter fields obtained by
    \gls{FOM}-\gls{IRGNM}, $q^{\text{\footnotesize FOM}}$, \gls{TR}-\gls{IRGNM}, $q^{\text{\footnotesize TR}}$, and their difference.
    }
    \label{fig:elasticity_alu_large_patch_stiff}
\end{figure}

\begin{figure}[H]
    \centering
    \includegraphics[width=\textwidth]{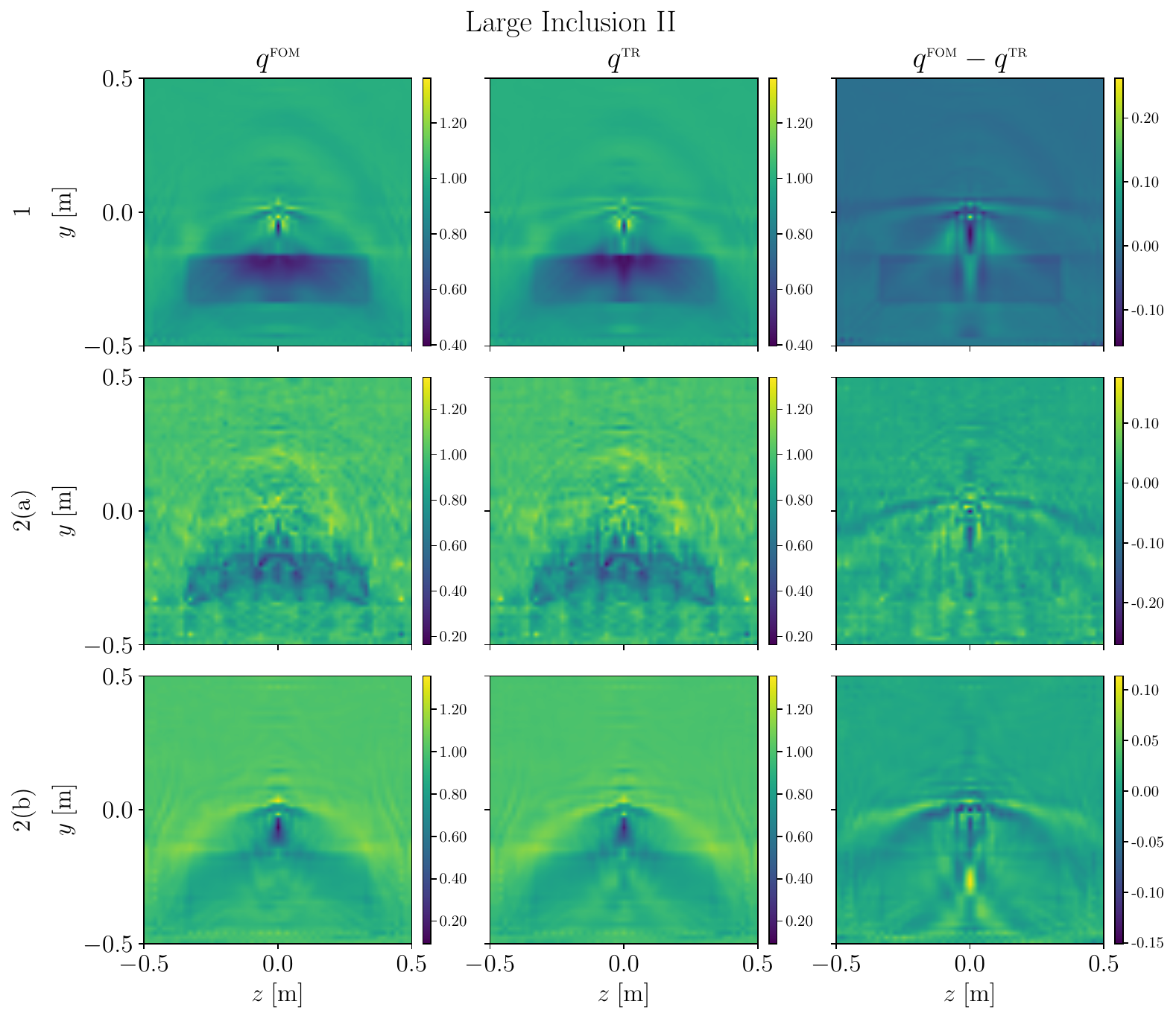}
    \caption{
    Reconstructed parameter field for the large inclusion II (softened case, $\tilde{q}=0.5$).
    Rows correspond to Experiments~1,~2(a), and~2(b), while columns show the parameter fields obtained by
    \gls{FOM}-\gls{IRGNM}, $q^{\text{\footnotesize FOM}}$, \gls{TR}-\gls{IRGNM}, $q^{\text{\footnotesize TR}}$, and their difference.
    }
    \label{fig:elasticity_alu_large_patch_soft}
\end{figure}

\FloatBarrier

\subsection{Point defects}

\begin{figure}[H]
    \centering
    \includegraphics[width=\textwidth]{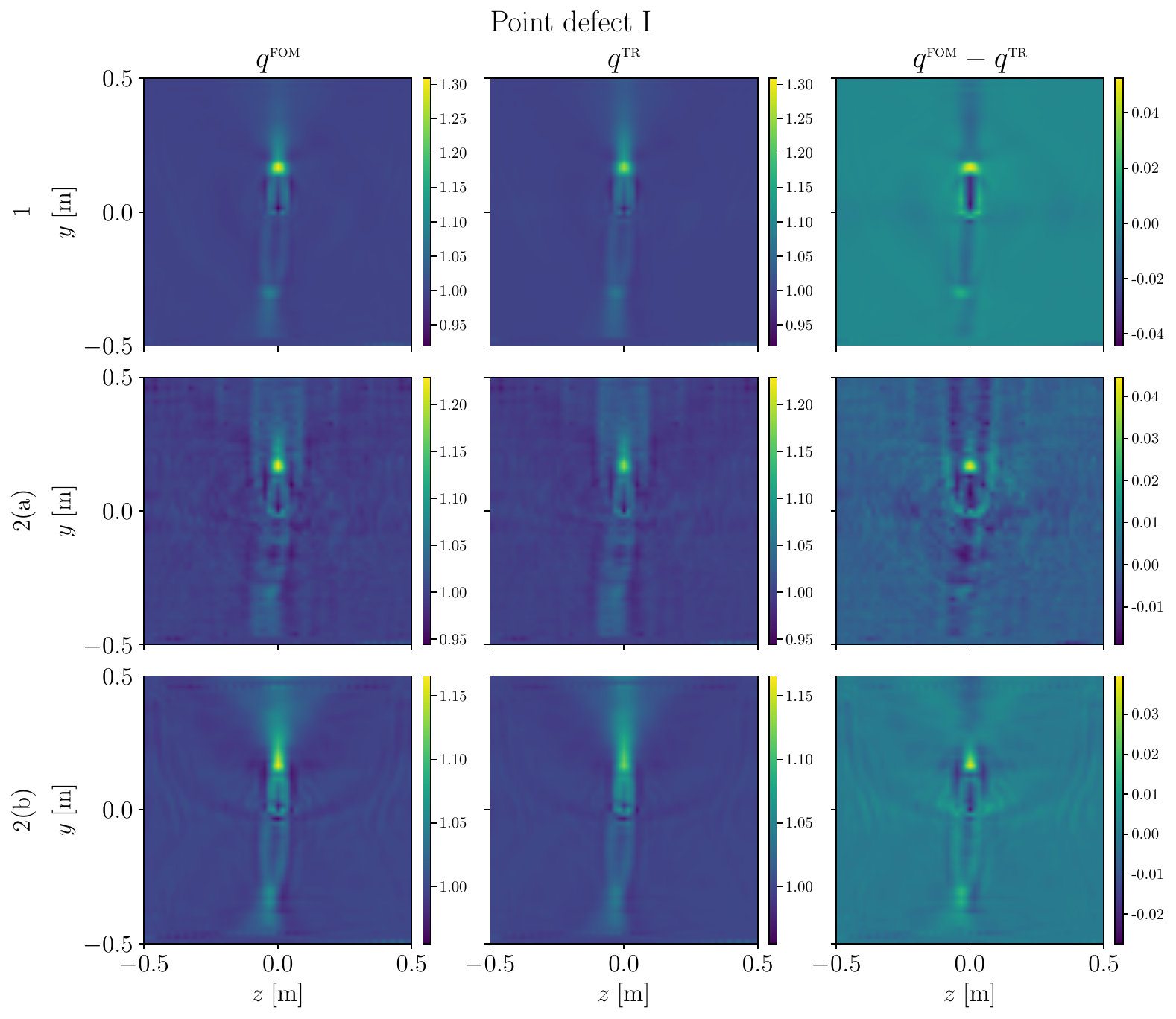}
    \caption{
    Reconstructed parameter field for point defect setup I. Rows correspond to Experiments~1,~2(a), and~2(b), while columns show the parameter fields obtained by
    \gls{FOM}-\gls{IRGNM}, $q^{\text{\footnotesize FOM}}$, \gls{TR}-\gls{IRGNM}, $q^{\text{\footnotesize TR}}$, and their difference.
    }
    \label{fig:elasticity_alu_point_defects_I}
\end{figure}

\end{document}